\documentclass[12pt,reqno]{amsart}
\usepackage{amsmath,amsthm,amssymb,amsfonts,amscd}
\usepackage{mathrsfs}
\usepackage{bbm}
\usepackage{bbding}
\usepackage{hyperref}
\usepackage{geometry}
\usepackage{color}
\usepackage{xcolor}

\usepackage{picture,epic}
\usepackage{tikz}

\usepackage{enumitem}

\numberwithin{equation}{section}

\theoremstyle{plain}
\newtheorem{theorem}{Theorem}[section]
\newtheorem{lemma}[theorem]{Lemma}

\newtheorem{proposition}[theorem]{Proposition}

\theoremstyle{definition}
\newtheorem{definition}[theorem]{Definition}

\theoremstyle{remark}
\newtheorem{remark}[theorem]{Remark}

\renewcommand{\Re}{\operatorname{Re}}
\renewcommand{\Im}{\operatorname{Im}}

\newcommand{\supp}{\operatorname{supp}}

\newcommand{\sym}{\operatorname{sym}}

\newcommand{\GL}{\operatorname{GL}}
\newcommand{\SL}{\operatorname{SL}}

\newcommand{\dd}{\mathrm{d}}

\makeatletter
\def\@tocline#1#2#3#4#5#6#7{\relax
  \ifnum #1>\c@tocdepth 
  \else
    \par \addpenalty\@secpenalty\addvspace{#2}%
    \begingroup \hyphenpenalty\@M
    \@ifempty{#4}{%
      \@tempdima\csname r@tocindent\number#1\endcsname\relax
    }{%
      \@tempdima#4\relax
    }%
    \parindent\z@ \leftskip#3\relax \advance\leftskip\@tempdima\relax
    \rightskip\@pnumwidth plus4em \parfillskip-\@pnumwidth
    #5\leavevmode\hskip-\@tempdima
      \ifcase #1
       \or\or \hskip 1em \or \hskip 2em \else \hskip 3em \fi%
      #6\nobreak\relax
    \hfill\hbox to\@pnumwidth{\@tocpagenum{#7}}\par
    \nobreak
    \endgroup
  \fi}
\makeatother

\begin{document}

\title[Effective decorrelation of Hecke eigenforms]
{Effective decorrelation of Hecke eigenforms}
\author{Bingrong Huang}
\address{Data Science Institute and School of Mathematics \\ Shandong University \\ Jinan \\ Shandong 250100 \\China}
\email{brhuang@sdu.edu.cn}


\date{\today}

\begin{abstract}
  In this paper, we prove effective quantitative decorrelation of values of two Hecke eigenforms as the weight goes to infinity.
  As consequences, we get an effective version of equidistribution of mass
  and
  zeros of certain linear combinations of  Hecke eigenforms.
\end{abstract}



\thanks{This work was supported by  the National Key R\&D Program of China (No. 2021YFA1000700) and NSFC (No. 12031008).}

\maketitle

\section{Introduction} \label{sec:Intr}

A central problem in the area of quantum chaos is to understand the value distribution of eigenfunctions on a Riemann surface $\mathbb{X}$ in the semi-classical limit.
Let $\phi$ denote a $L^2$-normalized eigenfunction. It is very interesting to understand the fluctuations of the probability measures $ |\phi(x)|^2 \dd x$.
The quantum ergodicity theorem (QE) proved by Shnirelman \cite{Shnirelman1974}, Colin de Verdi\`{e}re \cite{CdV1985}, and Zelditch \cite{Zelditch1987} for compact surfaces, and extended by Zelditch \cite{Zelditch1992} to noncompact surfaces such as modular surfaces,
implies that for a hyperbolic surface there exists a full density subsequence $\phi$ such that $\dd x$ is a weak* limit measure of $|\phi(x)|^2 \dd x$.

A typical example is the case when $\mathbb X=\Gamma\backslash \mathbb{H}$ and $\phi$ is a Hecke--Maass cusp form.
Here $\mathbb{H}=\{z=x+iy: x\in\mathbb{R}, \; y>0\}$ is the upper half plane and $\Gamma=\SL(2,\mathbb{Z})$ is the modular group.
Let $\dd \mu(z)=\frac{\dd x \dd y}{y^2}$ be hyperbolic measure.
The quantum unique ergodicity (QUE) conjectured by Rudnick--Sarnak \cite{RudnickSarnak1994} and proved by
Lindenstrauss \cite{Lindenstrauss2006} and Soundararajan  \cite{Soundararajan2010quantum} asserts that there are no exceptional subsequences, i.e., $|\phi(z)|^2 \dd\mu(z) \rightarrow \frac{3}{\pi}\dd\mu(z)$ as the eigenvalue of $\phi$ tends to infinity.
QUE (for Laplacians) concerns the diagonal matrix elements $\langle A \phi_j,\phi_j\rangle$ of pseudodifferential operators $A$ relative to an orthonormal basis $\{\phi_j\}$ of eigenfunctions of the Laplacian $\Delta=-y^2 (\partial^2 /\partial x^2+\partial^2 /\partial y^2)$ of the modular surface $\mathbb{X}$.

In \cite{Zelditch2004}, Zelditch observed the asymptotic vanishing of near off-diagonal matrix elements $\langle A \phi_i,\phi_j\rangle \, ( i\neq j)$ for eigenfunctions of QUE systems. Here ``near'' means the eigenvalue gaps tend to zero. In the modular surface case, we expect that much stronger results hold. For example, one may conjecture that all off-diagonal matrix elements tend to zero as the maximum of the eigenvalues tend to infinity. 

Holowinsky and Soundararajan \cite{Holowinsky,HS,Soundararajan2010weak} proved a holomorphic analog of the QUE conjecture, i.e., mass equidistribution.
Based on ideas in Iwaniec \cite{Iwaniec}, Lester, Matom\"aki, and Radziwi{\l\l} \cite{LMR} proved an effective version of quantitative equidistribution of mass  (EQQUE) for holomorphic Hecke eigenforms.
While  working on this project, we found Constantinescu \cite{Constantinescu} extended Holowinsky and Soundararajan's method to the decorrelation of two holomorphic Hecke eigenforms with almost equal weights.

In this paper, we extend the methods of Holowinsky and Soundararajan \cite{HS} and of Lester, Matom\"aki, and  Radziwi{\l\l} \cite{LMR} to prove an effective decorrelation of values of two Hecke  eigenforms, which can be viewed as an effective version of  a holomorphic  analog of Zelditch's result on near off-diagonal matrix elements.
As applications, we prove effective mass equidistribution of linear combinations of Hecke eigenforms ($J$-near Hecke eigenforms as in Definition \ref{def}), and equidistribution of their zeros.
In particular, we allow the number of Hecke eigenforms $J$ to grow slowly (a small power of the logarithm of the weight) as the weight tends to infinity.

\subsection{Effective decorrelation of values of Hecke eigenforms}

Let $\mathbb{H}=\{z=x+iy: x\in\mathbb{R}, \; y>0\}$ be the upper half plane and $\Gamma=\SL(2,\mathbb{Z})$ be the modular group.
Let $\mathbb X=\Gamma\backslash \mathbb{H}$ be the modular surface.
For an even integer $k\geq12$, let $S_k$ be the Hilbert space of all weight $k$ holomorphic cusp forms of level $1$. Let $H_k$ be a Hecke orthonormal basis of $S_k$. For $f,g\in H_k$, let $F=y^{k/2} f$ and $G=y^{k/2} g$. We have
\begin{equation}\label{eqn:innerproduct}
  \langle F,G\rangle := \int_{\mathbb{X}} F(z) \overline{G(z)} \frac{\dd x\dd y}{y^2}
  = \left\{
  \begin{array}{ll}
    1, & \textrm{if $f=g$}, \\
    0, & \textrm{if $f\neq g$.}
  \end{array} \right.
\end{equation}
Let $\mathcal{F}=\{z=x+iy: |x|\leq 1/2, \; |z|\geq1 \}$ be the standard fundamental domain of $\Gamma\backslash \mathbb{H}$.
Our main result in this paper is as follows.

\begin{theorem}\label{thm:decor}
  Let $k$ be a large even integer and  $1\leq M \leq \log k$.
  Let $f,g\in H_k$  and  $f\neq g$.
  Let $\psi\in \mathcal{C}_c^\infty(\mathbb X)$ satisfy
  \begin{equation}\label{eqn:psi_cond}
     y^{a+b} \frac{\partial^a}{\partial x^a} \frac{\partial^b}{\partial y^b}
     \psi(z) \ll_{a,b} M^{a+b} , \quad \textrm{for any $a,b\in \mathbb{Z}_{\geq0}$ and $z\in \mathcal{F}$}.
  \end{equation}
  Assume that $\psi|_{\mathcal{F}}$ has support contained in the interior of $\mathcal{F}$ and $\psi(x+iy)=0$ if $y\geq BM$ for some absolute constant $B>1$.
  Then we  have
  \[
    \langle \psi  F, G \rangle \ll_{\varepsilon} M^{5/3} (\log k)^{-\delta+\varepsilon},
  \]
  for  $\delta = 1.19\times 10^{-41}$.
\end{theorem}

\begin{remark}
  One of the main ingredients in the proof is Soundararajan and Thorner's recent wonderful result on weak subconvexity bounds of $L$-functions \cite{ST}, from which we get explicit dependence on $\psi$.
  Our exponent is much smaller than the diagonal case, for which Lester, Matom\"aki, and Radziwi{\l\l} \cite[Theorem 1.3]{LMR} give $0.0074$.\footnote{In \cite[P. 1625]{LMR} there is a misprint, one should take $\tfrac14(\lambda^2-1)^2\leq (\lambda-1)^2$ instead of $\tfrac14(\lambda^2-1)^2\leq \tfrac14(\lambda-1)^2$ which reduces the exponent from $0.0080...$ to $0.0074...$  as in Iwaniec \cite{Iwaniec}.}
  If we   assume the generalized Ramanujan conjecture (GRC) for $\GL(2)$ Hecke--Maass cusp forms, then we can take 
  $\delta=7/2-2\sqrt{3} \approx 0.0359$.
  One can prove effective decorrelation of Hecke--Maass cusp forms by our method under GRC.
   See \S \ref{subsec:ideas} below on discussions of the proof.
\end{remark}

%
%
%


As in the QUE case, we expect power saving upper bounds.
Let $f,g\in H_k$  and  $f\neq g$. Let $\psi\in \mathcal{C}_c^\infty(\mathbb X)$.
It is natural to conjecture that
  \begin{equation}\label{eqn:modularform_conj}
    \langle \psi F, G \rangle
    \ll_{\psi,\varepsilon} k^{-1/2+\varepsilon},
  \end{equation}
which is a consequence of the Grand Riemann hypothesis for the triple product $L$-functions $L(s,f\times g\times \phi)$ and the Rankin--Selberg $L$-functions $L(s,f\times g)$. Here $\phi$ is a Hecke--Maass cusp form for $\SL(2,\mathbb{Z})$.
We also expect the exponent $1/2$ is the best possible up to $\varepsilon$.

\begin{remark}
  To support the expectation of power saving with respect to the trivial bound $O(1)$, one may estimate a certain quantum variance by considering moments of central values of $L$-functions.
  Let $f\in H_k$. Let $\psi\in \mathcal{C}_c^\infty(\mathbb X)$ satisfy  \eqref{eqn:psi_cond} with $M=1$. Then one can at least prove
  \begin{equation} \label{eqn:QV-hol}
    \sum_{g\in H_k} | \langle \psi F, G \rangle |^2 \ll_\varepsilon  k^{2/3+\varepsilon}.
  \end{equation}
  In particular, we get $\langle \psi  F, G  \rangle
    \ll_\varepsilon  k^{-1/6+\varepsilon}$ for all but $O_\varepsilon(k^{1-\varepsilon})$ forms $g\in H_k$.
  To prove \eqref{eqn:QV-hol}, we can use Watson's formula, the approximate functional equation, the Petersson trace formula, 
  and bounds for $J$-Bessel functions.
\end{remark}

\subsection{Effective QUE and zero equidistribution}

Let $H_k$ be a Hecke orthonormal basis of $S_k$.
Holowinsky and Soundararajan \cite{HS} proved quantitative QUE for $f\in H_k$.
Lester, Matom\"aki, and Radziwi{\l\l} \cite{LMR} proved an effective version of quantitative QUE for $f\in H_k$.
As a consequence of Theorem \ref{thm:decor}, we can prove EQQUE for more general $f\in S_k$. We first introduce the following definition.

\begin{definition}\label{def}
  Let $J\geq1$ be a  positive integer.
  Define the set of \emph{$J$-near Hecke eigenforms} by
  \begin{multline*}
    H_k^{(J)} = \Big\{ f\in S_k : \textrm{there are distinct $f_j\in H_k$, ($j=1,2,\ldots,J$),}
     \\
     \textrm{and $(c_1,c_2,\ldots,c_J)\in \mathbb{C}^J$ such that $f= \sum_{1\leq j \leq J} c_j f_j$ }  \Big\}.
  \end{multline*}
\end{definition}
Note that any $f\in H_k^{(1)}$ is a Hecke eigenform.
Let $\boldsymbol{c}=(c_1,c_2,\ldots,c_J)\in \mathbb{C}^J$.
Note that $\|\boldsymbol{c}\|_2=(\sum_{1\leq j\leq J} |c_j|^2)^{1/2} = 1$ is equivalent to $\|f\|_2  = (\int_{\mathbb{X}} y^k |f(z)|^2 \dd \mu z)^{1/2} = 1$. We assume $\|f\|_2=1$.  Denote $F=y^{k/2}f$.
By EQQUE for Hecke eigenforms and Theorem \ref{thm:decor}, we get the following EQQUE result for $f\in H_k^{(J)}$.

\begin{theorem}\label{thm:QUE}
  Let $\psi\in \mathcal{C}_c^\infty(\mathbb X)$ be as in Theorem \ref{thm:decor}.
  Let $f\in H_k^{(J)}$ and $\|f\|_2=1$.  Then we have
  \[
    \langle \psi, |F|^2 \rangle = \frac{3}{\pi} \langle\psi,1\rangle + O_{\varepsilon}( J M^{5/3} (\log k)^{-\delta+\varepsilon}).
  \]
  In particular, when $\psi$ is fixed (i.e., $M=1$), then $f\in H_k^{(J)}$ satisfies QUE whenever $J\leq (\log k)^{\delta-\varepsilon}$.
\end{theorem}

\begin{remark}
  Since we have explicit dependence on $J$ and $M$ in the error term, we can allow both $J$ and $M$ to grow slowly as $k\rightarrow\infty$. For $M$ large, we get small scale mass equidistribution of $f\in H_k^{(J)}$.
\end{remark}

Let $f_1$ and $f_2$ be two distinct Hecke eigenforms in $H_k$. It is natural to ask the distribution of the solutions of $f_1(z)=c \cdot f_2(z)$ (counting with multiplicity) in a fundamental domain
for any $c\in \mathbb{C}$.
For the case $c=0$, this question asks the distribution of zeros of Hecke eigenforms.
The fact that $S_k$ is a Hilbert space makes this question be the same as to ask the zero distribution of the modular form $f=f_1-cf_2$, which is not a Hecke eigenform. Note that $f\in H_k^{(2)}$.

Rudnick \cite{Rudnick} proved that the zeros of Hecke eigenforms are equidistributed under QUE. 
Lester--Matom\"aki--Radziwi{\l\l} \cite{LMR} found an effective proof of Rudnick's theorem. Note that the key ingredient in their methods is QUE results for cusp forms, and equidistribution of zeros holds even for non Hecke eigenforms which satisfy QUE.
Hence we can prove equidistribution of zeros of $f\in H_k^{(J)}$ as $k\rightarrow \infty$
by their method together with Theorem \ref{thm:QUE}. We have the following result.

\begin{theorem}\label{thm:equidistribution}
  Let $k$ be a large even integer and $J\leq (\log k)^{\delta-\varepsilon}$ be a positive integer.
  As $k\rightarrow\infty$, the zeros of any  $f\in H_k^{(J)}\setminus \{0\}$ are equidistributed in $\mathbb{X}$ with respect to the normalized hyperbolic measure $(3/\pi) (\dd x \dd y /y^2)$.
\end{theorem}


Theorem \ref{thm:equidistribution} is a consequence of  \cite[Theorem 1.2]{Rudnick} and Theorem \ref{thm:QUE}. We will verify the condition that the order of $f$ at infinity $\nu_\infty(f)=o(k)$ in \S\ref{sec:zero}.\footnote{Recently, Kimmel \cite{Kimmel} noticed that QUE for $f$ implies $\nu_\infty(f)=o(k)$ by Rudnick's proof in \cite{Rudnick}.} 
One may also prove small scale equidistribution of the zeros of $f\in H_k^{(J)}$ as  in
Lester--Matom\"aki--Radziwi{\l\l}  \cite{LMR}.
See also \cite[\S9.3]{LR} for more discussions on equidistribution of zeros.

\subsection{Ideas of the proofs}\label{subsec:ideas}

To prove Theorem \ref{thm:decor}, we will use ideas from Holowinsky--Soundararajan, Iwaniec, and Lester--Matom\"aki--Radziwi{\l\l}.
As things currently stand, it is hard to use Soundararajan's original weak subconvexity bounds \cite{Soundararajan2010weak} to prove effective QUE and decorrelation, since the Ramanujan conjecture is not proven when the test function is a Hecke--Maass cusp form.
To overcome the difficulty from the unknown Ramanujan conjecture for Hecke--Maass cusp forms $\phi$, we will use the weak subconvexity bounds for the triple product $L$-functions $L(1/2,f\times g\times \phi)$ due to Soundararajan and Thorner \cite{ST}.
However, the saving is a very small power of $\log k$, so that the original approach of Holowinsky and Soundararajan may not work.
It is a surprise that  effective decorrelation can  still be proved from the small saving of Soundararajan and Thorner.
To succeed, we prove bounds by both Soundararajan's approach and Holowinsky's approach in forms that we   can use an optimization technique based on Iwaniec/Lester--Matom\"aki--Radziwi{\l\l} to achieve a small  power saving of $\log k$.

It seems difficult to avoid a use of weak subconvexity bounds of $L$-functions as Iwaniec and Lester--Matom\"aki--Radziwi{\l\l} did. In their Soundararajan's approach (see e.g. \cite[\S4.1]{LMR}), it is important to use the fact $f\times f=1\boxplus\sym^2 f$, so that it is $\lambda_f(p^2)^2$ matters. In the case $f\neq g$, if we do the same Cauchy--Schwarz, we face  $\lambda_f(p)^2 \lambda_g(p)^2$ or something similar.
In Holowinsky's approach, it does not work only when $|\lambda_f(p)|\approx |\lambda_g(p)|\approx 1$ for most  $p\leq k$, that is, $|\lambda_f(p^2)|\approx |\lambda_g(p^2)|\approx 0$ by the Hecke relations.
In the case $f=g$, we get  $\lambda_f(p^2)^2\approx 0$ which makes Soundararajan's approach (as in \cite[\S4.1]{LMR}) work.
While for the case $f\neq g$, we get
$\lambda_f(p)^2 \lambda_g(p)^2 \approx 1$, which is not what we want.
But we still expect Soundararajan's approach will work since we expect $|\lambda_f(p)|\approx |\lambda_g(p)|\approx 1$ is false from both the condition  $f\neq g$ and $f,g$ being cuspidal. It seems not easy to make full use of the condition $f\neq g$ in Iwaniec/Lester--Matom\"aki--Radziwi{\l\l}'s method.

Another difficulty occurs in Holowinsky's approach when $f\neq g$.
For the incomplete Eisenstein series, for the contribution from the zero Fourier coefficient $I_0(Y)$ in \eqref{eqn:I0}, we have $\lambda_f(n) \lambda_g(n) $ which is not always nonnegative as in the diagonal case we actually have $|\lambda_f(n)|^2$, so we can not use the asymptotic formula for $a_{\Psi,0}$ in Remark \ref{rmk:a0} as Holowinsky did. We solve this problem by making use of the full expression of $a_{\Psi,0}$ in Lemma \ref{lemma:FC}. See the proof of Lemma \ref{lemma:Eisenstein} for more details.%
\footnote{In Constantinescu \cite[Lemma 6.4]{Constantinescu}, there is a similar problem. One may not just apply the asymptotic formula for $a_0(y)$ as in the paper, since for the error term the factor $f(z)\overline{g(z)}$ is not always nonnegative. One may use our trick to complete the proof.}

\subsection{Plan for this paper}
The rest of this paper is organized as follows.
In \S \ref{sec:preliminaries}, we give a review of the theory of automorphic forms and Watson's formula, and prove a lemma on test functions that we will need later.
In \S \ref{sec:Soundararajan}, we use a variant of Soundararajan's approach to prove a bound for the inner product $\langle\psi F,G\rangle$. In \S \ref{sec:Holowinsky}, a variant of Holowinsky's approach is applied to prove another bound for the inner product $\langle\psi F,G\rangle$.
Then we combine these two bounds to prove Theorem \ref{thm:decor} by a local optimization in \S \ref{sec:proof}.
In \S \ref{sec:QUE}, we use Theorem \ref{thm:decor} to prove EQQUE (Theorem \ref{thm:QUE}).
Finally, 
In \S \ref{sec:zero},  we will show that EQQUE implies the bulk of zeros lies in the fundamental domain, and hence prove Theorem \ref{thm:equidistribution}.

\medskip
\textbf{Notation.}
Throughout the paper, $\varepsilon$ is an arbitrarily small positive number;
all of them may be different at each occurrence.
As usual, $e(x)=e^{2\pi i x}$.
We use $y\asymp Y$ to mean that $c_1 Y\leq |y|\leq c_2 Y$ for some positive constants $c_1$ and $c_2$.
The symbol $\ll_{a,b}$ denotes that the implied constant depends at most on $a$ and $b$.

\section{Preliminaries}\label{sec:preliminaries}

\subsection{Automorphic forms}

In this subsection, we briefly review the theory of $\GL(2)$ automorphic forms and state some basic facts which will be used in the proof of the theorems. \emph{Cf. } \cite[Chap. 14 and 15]{IwaniecKowalski2004analytic}.

Let $S_k$ be the Hilbert space of all weight $k$ holomorphic cusp forms for $\SL(2,\mathbb{Z})$. Let $H_k$ be a Hecke orthonormal basis of $S_k$.
For $f\in H_k$, we have
\[
  \|f\|_2^2 = \int_{\mathbb{X}} y^k |f(z)|^2 \dd \mu z = 1.
\]
Write the Fourier expansion of $f$ as
\[
  f(z) = a_f(1) \sum_{n\geq1} \lambda_f(n) n^{\frac{k-1}{2}} e(nz).
\]
Here $\lambda_f(n)\in \mathbb{R}$ is the eigenvalue of the $n$-th Hecke operator $T_n$.
By Deligne \cite{Deligne}, we have $|\lambda_f(n)|\leq \tau(n)=\sum_{d\mid n}1\ll n^\varepsilon$.
By the  theory of Rankin--Selberg, we have (see e.g. \cite[\S2.1]{BlomerKhanYoung})
\begin{equation}\label{eqn:firstcoeff}
  |a_f(1)|^2 = \frac{2\pi^2 (4\pi)^{k-1}}{\Gamma(k) L(1,\sym^2 f)}.
\end{equation}

Denote an orthonormal basis of Hecke--Maass cusp forms by $\{\phi_j\}_{j\geq1}$. For $\phi\in\{\phi_j\}_{j\geq1}$, we have $\Delta  \phi = (1/4+t_\phi^2) \phi$, where $t_\phi>1$ is the spectral parameter of $\phi$.

The Eisenstein series is defined by
\[
  E(z,s) := \sum_{\gamma\in\Gamma_\infty\backslash \Gamma} (\Im \gamma z)^s, \quad \Re(s)>1.
\]
This has meromorphic continuation to $\Re(s)\leq 1$.
The Fourier expansion is given by
\begin{equation}\label{eqn:FE-ES}
  E(z,s) = y^s + \frac{\theta(1-s)}{\theta(s)} y^{1-s}
  + \frac{2\sqrt{y}}{\theta(s)} \sum_{n\neq0} \tau_{s-1/2}(|n|) K_{s-1/2}(2\pi |n|y) e(nx),
\end{equation}
where $\theta(s)=\pi^{-s}\Gamma(s)\zeta(2s)$ and $\tau_{\nu}(n)=\sum_{ab=n}(a/b)^\nu$.
By Huang--Xu \cite{HuangXu}, we have the following sup-norm bounds
\begin{equation}\label{eqn:HX}
  E_t(z):=E(z,1/2+it) \ll_\varepsilon \sqrt{y} (1+|t|)^{3/8+\varepsilon} , \quad y\geq 1/2.
\end{equation}

To deal with a general observable $\psi$, we will use the following well-known decomposition so that we need to consider the Hecke--Maass cusp forms and the Eisenstein series.
\begin{lemma}[Selberg spectral decomposition]  \label{Lem:SSD}
  Let $\psi\in L^2(\mathbb{X})$. Then we have
  \[
    \psi(z) = \frac{3}{\pi} \langle \psi,1\rangle
    + \sum_{j\geq1} \langle \psi,\phi_j \rangle \phi_j(z)
    + \frac{1}{4\pi} \int_{\mathbb{R}} \langle \psi,E_t \rangle E_t(z) \dd t.
  \]
\end{lemma}

For the coefficients above, we will need the following bounds.
\begin{lemma}\label{lem:psi}
  Let $M\geq1$ be a positive number and $\psi\in \mathcal{C}_c^\infty(\mathbb X)$ satisfy \eqref{eqn:psi_cond}.
  Let $\phi$ be an $L^2$-normalized Hecke--Maass cusp form with the spectral parameter $t_\phi$.
  Then for all $A\geq0$, we have
  \[
    |\langle \psi,\phi \rangle| \ll_A  \left( \frac{M}{t_\phi}\right)^{A}
    \quad \textrm{and} \quad
    | \langle \psi, E_t \rangle | \ll_A \left( \frac{M}{1+|t|}\right)^{A}(1+|t|)^{3/8+\varepsilon}.
  \]
\end{lemma}

\begin{proof}
  Since $\psi\in \mathcal{C}_c^\infty(\mathbb X)$ satisfies \eqref{eqn:psi_cond} and $\Delta=-y^2 (\partial^2 /\partial x^2+ \partial^2/\partial y^2)$, we have
  \[
    \Delta^\ell \psi \ll_\ell M^{2\ell}, \quad \textrm{for all $\ell \in \mathbb{Z}_{\geq0}$}.
  \]
  The rest of the proof is the same as \cite[Lemma 4.3]{LMR} with \eqref{eqn:HX} replacing \cite[Lemma 4.2]{LMR}.
  Note that we have $\langle \psi,\phi \rangle \ll \int_{\mathbb{X}} |\phi(z)| \dd \mu z \ll1$ and
  $\langle \psi, E_t \rangle \ll \int_{\mathbb{X}} \sqrt{y} (1+|t|)^{3/8+\varepsilon} \dd \mu z \ll (1+|t|)^{3/8+\varepsilon} $. We allow $A\geq0$ instead of $A\geq2$ as in  \cite[Lemma 4.3]{LMR}.
\end{proof}


\subsection{Watson's formula}

In this subsection, we state some formulas which relate the integral representations of automorphic forms to their $L$-values.
We will follow \cite[\S2.3]{BlomerKhanYoung}.


Define
\[
  \mathcal{H}(k,t) = \frac{\pi^3 |\Gamma(k-1/2+it)|^2}{2 \Gamma(k)^2}.
\]
For $\phi$ being an even Hecke--Maass cusp form, Watson's formula gives
\begin{align}
  |\langle \phi F,G \rangle |^2
  & = \frac{\Lambda(1/2,f\times \bar g\times \phi)}{8 \Lambda(1,\sym^2  f) \Lambda(1,\sym^2 \bar g) \Lambda(1,\sym^2 \phi)} \nonumber \\
  & = \frac{L(1/2,{f}\times g\times \phi)} {L(1,\sym^2   f) L(1,\sym^2 g) L(1,\sym^2 \phi)} \mathcal{H}(k,t_\phi). \label{eqn:Watson-MHH}
\end{align}
The classical Rankin--Selberg theory computes the projection of $F\bar
G$ onto the Eisenstein series and the formula is
\begin{align} \label{eqn:RS}
  \frac{1}{4\pi} |\langle E_t F, G \rangle |^2
  & = \frac{1}{2\pi} \frac{|L(1/2+it,{f}\times \bar g)|^2} {L(1,\sym^2 f) L(1,\sym^2 \bar g) |\zeta(1+2it)|^2} \mathcal{H}(k,t).
\end{align}

%
%
%
%
%
%
%
%

%

%




\section{Soundararajan's approach} \label{sec:Soundararajan}

From this section, we start to prove effective quantitative decorrelation of holomorphic Hecke eigenforms (Theorem \ref{thm:decor}). We will combine ideas from Holowinsky--Soundararajan, Iwaniec, and Lester--Matom\"aki--Radziwi{\l\l}.

In this section we apply a variant of Soundararajan's approach.
We first consider the case when the test function is a Hecke--Maass cusp form. We have the following lemma, which is the main difference from previous works.
\begin{lemma} \label{lem:Maass}
  Let $f,g\in H_k$  and  $f\neq g$.
  Denote $F=y^{k/2} f$ and $G=y^{k/2} g$.
  Let $\phi$ be an even $L^2$-normalized Hecke--Maass cusp form for $\SL(2,\mathbb{Z})$ with the spectral parameter $t_\phi$. Assume $t_\phi\leq k^{1/2}$. Then we have
  \[
    \langle \phi F,G \rangle \ll_\varepsilon
       t_\phi^{1/2}  (\log k)^{-\delta_1+\varepsilon}
       \prod_{p\leq k} \left(1 - \frac{\tfrac12 \lambda_f(p^2) + \tfrac12  \lambda_g(p^2)}{p} \right),
  \]
  where $\delta_1=9.765625\times 10^{-21}$. 
  Furthermore, if we assume GRC for $\phi$, then  we can take $\delta_1=1/2$.
\end{lemma}

\begin{proof}
  By Watson's formula and Stirling's formula, we have
  \begin{equation}\label{eqn:Watson1}
    \langle \phi F,G \rangle \ll_\varepsilon
    \frac{1}{k^{1/2}} \frac{L(1/2,{f}\times g\times \phi)^{1/2}} {L(1,\sym^2   f)^{1/2} L(1,\sym^2 g)^{1/2} L(1,\sym^2 \phi)^{1/2}}.
  \end{equation}
  By the work of Hoffstein and Lockhart \cite{HL} we have
  \begin{equation}\label{eqn:HL}
    L(1,\sym^2 \phi)^{-1} \ll \log t_\phi.
  \end{equation}
  By Holowinsky and Soundararajan \cite[Lemma 2]{HS} we get
  \begin{equation}\label{eqn:HS}
    \begin{split}
      L(1,\sym^2   f)^{-1} \ll (\log\log k)^3  \prod_{p\leq k} \left(1 - \frac{  \lambda_f(p^2) }{p} \right), \\
    L(1,\sym^2   g)^{-1} \ll (\log\log k)^3  \prod_{p\leq k} \left(1 - \frac{  \lambda_g(p^2) }{p} \right).
    \end{split}
  \end{equation}

  Now we need to bound the $L$-function $L(1/2,{f}\times g\times \phi)$.
  The analytic conductor $\mathfrak{C}$ of this $L$-function satisfies $\mathfrak{C} \asymp k^4 t_\phi^4$. Since we assume $f$ and $g$ are orthogonal, the automorphic representation $\pi_1$ corresponding to ${f}\times g$ is a $\GL(4)$ cuspidal automorphic representation. Let $\pi_2$ be the $\GL(2)$ cuspidal automorphic representation corresponding to $\phi$. Then by Soundararajan and Thorner \cite[Corollary 2.7]{ST} we obtain
  \begin{equation}\label{eqn:ST}
    L(1/2,{f}\times g\times \phi) \ll  k t_\phi  (\log k)^{-1/(10^{17} \times 8^3)}.
  \end{equation}
  If we further assume the GRC for $\phi$, then by Soundararajan \cite{Soundararajan2010weak} we get
  \begin{equation*}
    L(1/2,{f}\times g\times \phi) \ll  k t_\phi  (\log k)^{-1+\varepsilon}.
  \end{equation*}
  Combining \eqref{eqn:Watson1}, \eqref{eqn:HL}, \eqref{eqn:HS}, and \eqref{eqn:ST}, we complete the proof.
\end{proof}

Next, we consider the case when the test function is an Eisenstein series.

\begin{lemma}\label{lem:Eisenstein}
  Let $f,g\in H_k$  and  $f\neq g$.
  Denote $F=y^{k/2} f$ and $G=y^{k/2} g$.
  Let $E(z,s)$ be the Eisenstein series for $\SL(2,\mathbb{Z})$ and $E_t=E(\cdot,1/2+it)$.
  Assume $t\in\mathbb{R}$ and $|t|\leq k^{1/2}$. Then we have
  \[
    \langle E_t F,G \rangle \ll_\varepsilon
      (1+|t|)^{1/2}  (\log k)^{-1+\varepsilon}
       \prod_{p\leq k} \left(1 - \frac{\tfrac12 \lambda_f(p^2) + \tfrac12  \lambda_g(p^2)}{p} \right).
  \]
\end{lemma}

\begin{proof}
  By the theory of Rankin--Selberg \eqref{eqn:RS} and  Stirling's formula, we have
  \[
    \langle E_t F,G \rangle
    \ll_\varepsilon \frac{1}{k^{1/2}}
    \frac{|L(1/2+it,f\times  g)|} {L(1,\sym^2 f)^{1/2} L(1,\sym^2 g)^{1/2} |\zeta(1+2it)|}.
  \]
  It is well known that
  \[ \zeta(1+2it)^{-1} \ll \log (1+|t|).  \]
  Note that ${f}\times g$ is a $\GL(4)$ cuspidal automorphic form which satisfies the generalized Ramanujan conjecture. The analytic conductor of $L(1/2+it,f\times  g) $ is of size $k^2 (1+|t|)^2$. By Soundararajan \cite[Theorem 1]{Soundararajan2010weak}, we have
  \[
    L(1/2+it,f\times  g) \ll \frac{k^{1/2} (1+|t|)^{1/2}}{(\log k)^{1-\varepsilon}}.
  \]
  Together with \eqref{eqn:HS}, we complete the proof.
\end{proof}

Now we are ready to consider the test function $\psi$ as in Theorem \ref{thm:decor}.

\begin{proposition} \label{prop:Sound}
  With the notation and assumptions of Theorem \ref{thm:decor}, we have
  \[
    \langle \psi F, G \rangle
    = O_\varepsilon \left(  M^{3/2} (\log k)^\varepsilon
       \prod_{p\leq k} \left(1 - \frac{\tfrac12 \lambda_f(p^2) + \tfrac12  \lambda_g(p^2)+\delta_1}{p} \right) \right) ,
  \]
  where $\delta_1=9.76563\times 10^{-21}$.
  Furthermore, if we assume GRC for $\GL(2)$ Hecke--Maass cusp forms, then  we can take $\delta_1=1/2$.
\end{proposition}

\begin{proof}
  The proof is similar to \cite[Lemma 4.4]{LMR}.
  By the Selberg spectral decomposition (Lemma \ref{Lem:SSD}) we get
  \[
    \langle \psi F, G \rangle
    = \frac{3}{\pi} \langle \psi,1\rangle   \langle  F, G \rangle
    + \sum_{\phi \; \rm cusp} \langle \psi,\phi \rangle  \langle \phi  F, G \rangle
    + \frac{1}{4\pi} \int_{\mathbb{R}} \langle \psi,E_t \rangle \langle E_t   F, G \rangle \dd t.
  \]
  Since $f\neq g$, we have $\langle  F, G \rangle=0$.
  By Watson's formula, the lower bounds for $L$-values on the  $1$-line, and the convexity bounds of $L$-values on the critical line, for $t_\phi\geq k^{\varepsilon}$ and real $|t|\geq k^{\varepsilon}$, we have
  \[
    \langle \phi  F, G \rangle \ll ((k+t_\phi)^4 t_\phi^4)^{1/8+\varepsilon} k^{-1/2}
    \ll t_\phi^{1+\varepsilon}
  \]
  and
  \[
    \langle E_t   F, G \rangle \ll ((k+|t|)^2 (1+|t|)^{2} )^{1/4+\varepsilon} k^{-1/2}
    \ll (1+|t|)^{1+\varepsilon}.
  \]
  By Lemma \ref{lem:psi}, for any $A>0$ we have
  \begin{equation*}
    \sum_{t_\phi \geq k^\varepsilon} |\langle \psi,\phi \rangle  \langle \phi  F, G \rangle |
    + \frac{1}{4\pi} \int_{|t|\geq k^\varepsilon} | \langle \psi,E_t \rangle \langle E_t   F, G \rangle | \dd t  \ll_A k^{-A}.
  \end{equation*}
  For $t_\phi\in [M(\log k)^\varepsilon,k^\varepsilon]$ and real $|t|\in [M(\log k)^\varepsilon,k^\varepsilon]$, by Lemmas \ref{lem:Maass} and \ref{lem:Eisenstein}, we get
  \[
    \langle \phi  F, G \rangle \ll t_\phi^{1/2}  \log k, \quad
    \langle E_t   F, G \rangle \ll (1+|t|)^{1/2} \log k.
  \]
  Hence by Lemma \ref{lem:psi}, for any $A>0$ we get
  \begin{multline*}
    \sum_{t_\phi\in [M(\log k)^\varepsilon,k^\varepsilon]} | \langle \psi,\phi \rangle  \langle \phi  F, G \rangle |
    + \frac{1}{4\pi} \int_{|t|\in [M(\log k)^\varepsilon,k^\varepsilon]}
    | \langle \psi,E_t \rangle \langle E_t   F, G \rangle|  \dd t
    \\
     \ll_A
     \sum_{t_\phi\in [M(\log k)^\varepsilon,k^\varepsilon]} \frac{M^A}{t_\phi^A}   t_\phi^{1/2}  \log k
    + \int_{|t|\in [M(\log k)^\varepsilon,k^\varepsilon]}
    \frac{M^A}{|t|^A} |t|^{1/2}   |t|^{1/2}  \log k \; \dd t.
  \end{multline*}
  It is well known that  $\sum_{T < t_\phi \leq  2T}1 \ll T^2$ (\emph{cf.} \cite[Chap. 15]{IwaniecKowalski2004analytic}). Hence  we obtain
  \begin{align*}
    \sum_{t_\phi\in [M(\log k)^\varepsilon,k^\varepsilon]} & | \langle \psi,\phi \rangle \langle \phi  F, G \rangle |
    + \frac{1}{4\pi} \int_{|t|\in [M(\log k)^\varepsilon,k^\varepsilon]}
    | \langle \psi,E_t \rangle \langle E_t   F, G \rangle|  \dd t \\
    & \ll_A
     M^A (\log k)^3 \max_{M(\log k)^\varepsilon \leq T\leq k^\varepsilon} T^{5/2-A}
    + M^A ( \log k)  \int_{ M(\log k)^\varepsilon}^{k^\varepsilon}
     t^{1-A}   \dd t \\
    & \ll_A
     M^A (\log k)^3  M^{5/2-A} (\log k)^{\varepsilon(5/2-A)}
     \ll_A (\log k)^{\varepsilon(5/2-A) +11/2}.
  \end{align*}
  Here we take $A\geq3$ being sufficiently large and we have used the assumption $M\leq \log k$. Hence the contribution from those terms is $O_{\varepsilon, A}((\log k)^{-A})$ for any $A>0$.  So we have
  \begin{equation}\label{eqn:<psiFG><<}
    \langle \psi F, G \rangle
    =    \sum_{t_\phi \leq M (\log k)^\varepsilon} \langle \psi,\phi \rangle  \langle \phi  F, G \rangle
    + \frac{1}{4\pi} \int_{ |t| \leq M (\log k)^\varepsilon } \langle \psi,E_t \rangle \langle E_t   F, G \rangle \dd t + O_{\varepsilon, A}((\log k)^{-A}).
  \end{equation}

  Parseval's identity gives
  \[
    \|\psi\|_2^2 = \frac{3}{\pi} |\langle \psi,1 \rangle|^2
    + \sum_{\phi} | \langle \psi, \phi\rangle |^2
    + \frac{1}{4\pi} \int_{\mathbb{R}} | \langle \psi, E_t \rangle|^2 \dd t.
  \]
  By \eqref{eqn:psi_cond} we have $\|\psi\|_2^2 = \int_{\mathbb{X}}|\psi(z)|^2 \dd\mu z \ll 1$. For Maass cusp forms, by Cauchy--Schwarz and Lemma \ref{lem:Maass}, we get
  \begin{equation}\label{eqn:<psiFG><<Maass}
     \begin{split}
     \sum_{t_\phi \leq M (\log k)^\varepsilon} & \langle \psi,\phi \rangle \langle \phi  F, G \rangle
     \ll \left( \sum_{t_\phi \leq M (\log k)^\varepsilon} |\langle \psi,\phi \rangle|^2 \right)^{1/2}
     \left( \sum_{t_\phi \leq M (\log k)^\varepsilon} |\langle \phi  F, G \rangle|^2 \right)^{1/2}
     \\
     &\ll \|\psi\|_2 \left( \sum_{t_\phi \leq M (\log k)^\varepsilon}
    t_\phi (\log k)^{-2\delta_1+\varepsilon}
       \prod_{p\leq k} \left(1 - \frac{  \lambda_f(p^2) +   \lambda_g(p^2)}{p} \right) \right)^{1/2} \\
    & \ll M^{3/2} (\log k)^\varepsilon
       \prod_{p\leq k} \left(1 - \frac{\tfrac12 \lambda_f(p^2) + \tfrac12  \lambda_g(p^2)+\delta_1}{p} \right).
    \end{split}
  \end{equation}
  Here we have used the fact $\prod_{p\leq k} (1-\frac{\delta}{p}) \asymp (\log k)^{-\delta}$ for any $\delta\in\mathbb{R}$.

  For Eisenstein series,  by Cauchy--Schwarz and Lemma \ref{lem:Eisenstein}, we get
  \begin{equation}\label{eqn:<psiFG><<Eisenstein}
    \begin{split}
     & \hskip -0.5cm \int_{ |t| \leq M (\log k)^\varepsilon } \langle \psi,E_t \rangle \langle E_t   F, G \rangle \dd t \\
     & \ll
     \left( \int_{ - M (\log k)^\varepsilon }^{M (\log k)^\varepsilon } |\langle \psi,E_t \rangle|^2  \dd t \right)^{1/2}
     \left(  \int_{ - M (\log k)^\varepsilon }^{M (\log k)^\varepsilon } | \langle E_t   F, G \rangle|^2  \dd t \right)^{1/2}
     \\
    &\ll \|\psi\|_2 \left(  \int_{ - M (\log k)^\varepsilon }^{M (\log k)^\varepsilon } (1+|t|)  \dd t \right)^{1/2}  (\log k)^{-1+\varepsilon}
       \prod_{p\leq k} \left(1 - \frac{\tfrac12 \lambda_f(p^2) + \tfrac12  \lambda_g(p^2)}{p} \right) \\
    &\ll M (\log k)^\varepsilon
       \prod_{p\leq k} \left(1 - \frac{\tfrac12 \lambda_f(p^2) + \tfrac12  \lambda_g(p^2)+1}{p} \right).
    \end{split}
  \end{equation}
  Combining \eqref{eqn:<psiFG><<}, \eqref{eqn:<psiFG><<Maass}, and  \eqref{eqn:<psiFG><<Eisenstein}, we complete the proof.
\end{proof}

\section{Holowinsky's approach} \label{sec:Holowinsky}

In this section, we apply a variant of Holowinsky's method \cite{Holowinsky} to deal with the inner product $\langle \psi F, G \rangle$.
The general strategy is to expand $\psi$ into a linear combination of incomplete Poincar\'e series. We will also use ideas from Lester--Matom\"aki--Radziwi{\l\l} \cite{LMR}.

Let $\mathcal{F}$ be the standard fundamental domain of $\Gamma\backslash \mathbb{H}$.
Let $\check{\psi}$ be the function on $\mathbb{H}$ to $\mathbb{C}$ such that $\check{\psi}(z)=\psi(z)$ if $z\in \mathcal{F}$ and $\check\psi(z)=0$ if $z\notin\mathcal{F}$.
Let $\check\Psi$ be the extension of $\check\psi$ to $\mathbb{H}$ by $\Gamma_\infty=\{ (\begin{smallmatrix}  1 & n \\  0 & 1 \end{smallmatrix}):n\in \mathbb{Z}\}$ periodicity.
Define
\[
  \Psi_m(y) = \int_{-1/2}^{1/2} \check\Psi(x+iy) e(-mx) \dd x.
\]
We have the usual Fourier expansion of $\check\Psi$
\[
  \check\Psi(x+iy) = \sum_{m\in \mathbb{Z}} \Psi_m(y) e(mx).
\]
Note that
\[
  \psi(z) = \sum_{\gamma\in\Gamma_\infty\backslash \Gamma} \check\Psi(\gamma z).
\]
For $\Psi\in \mathcal{C}^{\infty}(\mathbb{R})$, we define the incomplete Poincar\'e series
\[
  P_n(z,\Psi) = \sum_{\gamma\in\Gamma_\infty\backslash \Gamma} e(n \Re(\gamma z)) \Psi(\Im(\gamma z)).
\]
Thus we have
\[
  \psi(z)= \sum_{\gamma\in\Gamma_\infty\backslash \Gamma}
  \sum_{m\in \mathbb{Z}} \Psi_m( \Im \gamma z) e(m \Re \gamma z)
  = \sum_{m\in\mathbb{Z}} P_m(z,\Psi_m).
\]
Hence we obtain
\begin{equation}\label{eqn:<psiFG>=PS}
  \langle \psi F, G \rangle = \sum_{m\in\mathbb{Z}}  \langle P_m(\cdot,\Psi_m) F, G \rangle .
\end{equation}

Note that for any $z\in \mathcal{F}$, we have $\Im z\geq \sqrt{3}/2$.
By \eqref{eqn:psi_cond} and integration by parts, we have $\Psi_m(y) \ll_A (M/|m|)^{A}$.
Hence by unfolding, for $|m|\geq1$, we have
\begin{align*}
  \langle P_m(\cdot,\Psi_m) F, G \rangle
  & = \int_{0}^{\infty} \int_{-1/2}^{1/2} y^k f(z) \bar{g}(z) e(mx) \Psi_m(y) \frac{\dd x\dd y }{y^2} \\
  & \ll_A \int_{\sqrt{3}/2}^{\infty} \int_{-1/2}^{1/2} y^k |f(z)| |g(z)|  \frac{\dd x\dd y }{y^2} \left(\frac{M}{|m|}\right)^A .
\end{align*}
By the Cauchy--Schwarz inequality, we get
\begin{align*}
  \langle P_m(\cdot,\Psi_m) F, G \rangle
  & \ll_A \int_{\sqrt{3}/2}^{\infty} \int_{-1/2}^{1/2} y^k (|f(z)|^2 + |g(z)|^2)  \frac{\dd x\dd y }{y^2} \left(\frac{M}{|m|}\right)^A
  \ll_A \left(\frac{M}{|m|}\right)^A.
\end{align*}
Hence for any $A>1$, we have
\begin{align*}
  \sum_{|m|\geq M (\log k)^\varepsilon}  \langle P_m(\cdot,\Psi_m) F, G \rangle
  & \ll_A M^A \sum_{m\geq M (\log k)^\varepsilon}  m^{-A} \\
  & \ll_A M^A (M(\log k)^\varepsilon)^{1-A} \ll_A M (\log k)^{(1-A)\varepsilon}.
\end{align*}
Since we assume $M\leq \log k$, the above is $O_{\varepsilon,A}((\log k)^{-A})$ for any $A>0$.
By \eqref{eqn:<psiFG>=PS} we get
\begin{equation}\label{eqn:<psiFG><<PS}
  \langle \psi F, G \rangle = \sum_{|m|\leq M (\log k)^\varepsilon}  \langle P_m(\cdot,\Psi_m) F, G \rangle + O_{\varepsilon,A}((\log k)^{-A}).
\end{equation}
We still need to estimate $\langle P_m(\cdot,\Psi_m) F, G \rangle $ for $|m|\leq M (\log k)^\varepsilon$. We first consider the case $m\neq0$ and we have the following lemma.

\begin{lemma}\label{lemma:Poincare}
  Let $1\leq M\leq \log k$. Let $1\leq |m| \leq  M (\log k)^\varepsilon$.
  Let $\Psi:\mathbb{R}\rightarrow\mathbb{C}$ be a smooth compactly supported function such that $y^\ell \Psi^{(\ell)}(y) \ll_\ell M^\ell$ for all $\ell\in\mathbb{Z}_{\geq0}$ and $\supp \Psi\subset [1/2,B M]$ for some absolute constant $B>1$.
  Then we have
  \[
    \langle P_m(\cdot,\Psi) F, G \rangle \ll_{\varepsilon} (\log k)^\varepsilon
    \prod_{p\leq k} \left( 1-\frac{\tfrac12 (|\lambda_f(p)|-1)^2+\tfrac12 (|\lambda_g(p)|-1)^2}{p} \right).
  \]
\end{lemma}

\begin{proof}
  By the unfolding method, we have
  \[
    \langle P_m(\cdot,\Psi) F, G \rangle
    = \int_{0}^{\infty} \int_{-1/2}^{1/2} y^k f(z) \bar{g}(z) \Psi(y) e(mx) \frac{\dd x\dd y}{y^2}.
  \]
  By the Fourier expansions, we get
  \begin{align*}
    \langle P_m(\cdot,\Psi) F, G \rangle
    & = \int_{0}^{\infty} \int_{-1/2}^{1/2} y^k  a_f(1) \sum_{n_1\geq1} \lambda_f(n_1) n_1^{\frac{k-1}{2}} e(n_1 x) \\
    & \qquad \cdot
    \overline{a_g(1)} \sum_{n_2\geq1} \overline{\lambda_g(n_2)} n_2^{\frac{k-1}{2}} e(-n_2 x) \exp(-2\pi (n_1+n_2) y)
     \Psi(y) e(mx) \frac{\dd x\dd y}{y^2} \\
    & = a_f(1)\overline{a_g(1)} \sum_{\substack{n\geq1 \\ n+m\geq1}}  \lambda_f(n) \lambda_g(n+m)  n^{\frac{k-1}{2}} (n+m)^{\frac{k-1}{2}} \\
    & \hskip 90pt  \cdot \int_{0}^{\infty}   y^{k-2}  \exp(-2\pi (2n+m) y)
     \Psi(y)   \dd y .
  \end{align*}
  Now we will use the method in Luo--Sarnak \cite[\S2]{LS} to deal with the above sum and integral. Define the Mellin transform of $\Psi$ as
  \[
    \tilde\Psi(s) = \int_{0}^{\infty} \Psi(y) y^{s-1} \dd y.
  \]
  Then $\tilde\Psi(s)$ is entire as $\Psi\in \mathcal{C}_c^\infty(\mathbb{R}_{>0})$, and for any $\ell\in\mathbb{Z}_{\geq0}$ and $\Re(s)<0$ we have $\tilde\Psi(s) \ll_\ell \frac{M^\ell}{(1+|s|)^\ell}$ by integration by parts $\ell$  times. Hence for $\Re(s)<0$ we have
  \begin{equation}\label{eqn:tildePsi<<}
    \tilde\Psi(s) \ll_A \left( \frac{M}{1+|s|}\right)^{A}, \quad \textrm{ for any $A\geq0$.}
  \end{equation}
  Mellin inversion gives us
  \[
    \Psi(x) = \frac{1}{2\pi i} \int_{(-2)} \tilde{\Psi}(s) x^{-s} \dd s
    =\frac{1}{2\pi i} \int_{(2)} \tilde{\Psi}(-s) x^{s} \dd s, \quad x>0.
  \]
  Hence
  \[
    \int_{0}^{\infty}   y^{k-2}  \exp(-2\pi (2n+m) y)
     \Psi(y)   \dd y
     = \frac{1}{2\pi i} \int_{(2)} \tilde{\Psi}(-s) \int_{0}^{\infty}    y^{k+s-2}  \exp(-2\pi (2n+m) y) \dd y \dd s.
  \]
  Note that $\Gamma(s)=\int_{0}^{\infty} x^{s-1} \exp(-x) \dd x$, $\Re(s)>0$. We get
  \begin{equation}\label{eqn:integral}
    \int_{0}^{\infty}   y^{k-2}  \exp(-2\pi (2n+m) y)
     \Psi(y)   \dd y
     = \frac{1}{2\pi i} \int_{(2)} \tilde{\Psi}(-s) \frac{\Gamma(s+k-1)}{(2\pi (2n+m))^{s+k-1}} \dd s.
  \end{equation}
  Hence
  \begin{align*}
    \langle P_m(\cdot,\Psi) F, G \rangle
    & = a_f(1)\overline{a_g(1)} \frac{\Gamma(k-1)}{(4\pi)^{k-1}} \sum_{\substack{n\geq1 \\ n+m\geq1}}  \lambda_f(n) \lambda_g(n+m)  \frac{(n(n+m))^{\frac{k-1}{2}}}{(n+m/2)^{k-1}}  \\
    & \hskip 90pt  \cdot \frac{1}{2\pi i} \int_{(2)} \tilde{\Psi}(-s) \frac{\Gamma(s+k-1)}{\Gamma(k-1)} \frac{1}{(4\pi (n+m/2))^{s}} \dd s.
  \end{align*}
  By Stirling's formula, for $0<a\leq \Re(s)\leq b$ we have (see \cite[Eq. (2.3)]{LS})
  \[
    \frac{\Gamma(s+k-1)}{\Gamma(k-1)}
    = (k-1)^s \left( 1+O_{a,b}\Big(\frac{(|1+|s|)^2}{k}\Big)\right).
  \]
  By moving the contour to $\Re(s)=b$ with sufficiently large $b>0$, the contribution from $n\geq k^{1+\varepsilon}$ is bounded by
  \begin{multline*}
    O_{b} \left( k^{-1+\varepsilon} \sum_{n\geq k^{1+\varepsilon}} n^\varepsilon \int_{(b)} \frac{M^4}{(1+|s|)^4}  k^b \left( 1+\frac{(|1+|s|)^2}{k}\right) n^{-b} |\dd s|\right) \\
    = O_b \left( k^{-1+\varepsilon} k^2  k^{\varepsilon (2-b)} M^4 \right)
    = O_{\varepsilon,A} (k^{-A}),
  \end{multline*}
  for any $A>0$. Here we have used $M\ll \log k$. Hence we obtain
  \begin{align*}
    \langle P_m(\cdot,\Psi) F, G \rangle
    & = a_f(1)\overline{a_g(1)} \frac{\Gamma(k-1)}{(4\pi)^{k-1}}
    \sum_{\max(1,1-m)\leq n\leq k^{1+\varepsilon}}  \lambda_f(n) \lambda_g(n+m)  \left( \frac{\sqrt{n(n+m)}}{n+m/2}\right)^{k-1}  \\
    & \hskip 10pt  \cdot \frac{1}{2\pi i} \int_{(\varepsilon)} \tilde{\Psi}(-s)  \left( 1+O_{\varepsilon}\Big(\frac{(|1+|s|)^2}{k}\Big)\right) \frac{(k-1)^s}{(4\pi (n+m/2))^{s}} \dd s + O_{\varepsilon,A} (k^{-A}).
  \end{align*}
  For $\max(1,1-m)\leq n$, by the Cauchy--Schwarz inequality we have $\sqrt{n(n+m)} \leq n+m/2$.
  By \eqref{eqn:firstcoeff} the contribution from the error terms in Gamma functions is
  bounded by
  \[
    O_\varepsilon \left( k^{-1+\varepsilon} \sum_{n\leq k^{1+\varepsilon}} n^\varepsilon \int_{(\varepsilon)}  \frac{M^4}{(1+|s|)^4} \frac{(|1+|s|)^2}{k} k^\varepsilon |\dd s|\right) = O_\varepsilon \left( k^{-1+\varepsilon} M^4 \right)
    = O_\varepsilon \left( k^{-1+\varepsilon} \right).
  \]
  Note that
  \[ \frac{1}{2\pi i} \int_{(\varepsilon)} \tilde{\Psi}(-s)  \frac{(k-1)^s}{(4\pi (n+m/2))^{s}} \dd s = \Psi \left( \frac{k-1}{4\pi (n+m/2)} \right).\]
  By the assumptions of this lemma, we only need to consider  $k^{1-\varepsilon} \ll n\ll k$ and $m\ll k^\varepsilon$, in which case we have
  \[  \left( \frac{\sqrt{n(n+m)}}{n+m/2}\right)^{k-1}= 1+O(k^{-1+\varepsilon}). \]
  Thus we get
  \begin{align*}
    \langle P_m(\cdot,\Psi) F, G \rangle
    & = a_f(1)\overline{a_g(1)} \frac{\Gamma(k-1)}{(4\pi)^{k-1}}    \\
    & \hskip 10pt  \cdot
    \sum_{\max(1,1-m)\leq n\leq k^{1+\varepsilon}}  \lambda_f(n) \lambda_g(n+m)
    \Psi \left( \frac{k-1}{4\pi (n+m/2)} \right) + O_{\varepsilon} (k^{-1+\varepsilon}).
  \end{align*}
  Note that $\supp \Psi\subset [1/2,B M]$. Together with \eqref{eqn:firstcoeff} we have
  \begin{align*}
    \langle P_m(\cdot,\Psi) F, G \rangle
    & \leq   \frac{2\pi^2}{k L(1,\sym^2 f)^{1/2} L(1,\sym^2 g)^{1/2}}
    \sum_{n\leq k }  |\lambda_f(n) \lambda_g(n+m)|  + O_{\varepsilon} (k^{-1+\varepsilon}).
  \end{align*}
  By Holowinsky \cite[Theorem 1.2]{Holowinsky} we get
  \[
     \sum_{n\leq k }  |\lambda_f(n) \lambda_g(n+m)|
     \ll k (\log k)^{\varepsilon-2} \prod_{p\leq \exp(\frac{\log k}{\varepsilon \log\log k})} \left(1+\frac{|\lambda_f(p)|}{p} \right) \left(1+\frac{|\lambda_g(p)|}{p} \right).
  \]
  By Degline's bounds and Mertens' theorem we know
  $$\prod_{ \exp(\frac{\log k}{\varepsilon \log\log k}) < p \leq k} \left(1+\frac{|\lambda_f(p)|}{p} \right) \leq \prod_{ \exp(\frac{\log k}{\varepsilon \log\log k}) < p \leq k} \left(1+\frac{2}{p} \right) \ll (\log\log k)^2.$$
  Hence we obtain
  \begin{equation}\label{eqn:Holowinsky}
     \sum_{n\leq k }  |\lambda_f(n) \lambda_g(n+m)|
     \ll k (\log k)^{\varepsilon} \prod_{p\leq k}
     \left(1+\frac{|\lambda_f(p)|-1}{p} \right) \left(1+\frac{|\lambda_g(p)|-1}{p} \right),
  \end{equation}
  and together with \eqref{eqn:HS} we have
  \begin{align*}
    \langle P_m(\cdot,\Psi) F, G \rangle
    & \leq  (\log k)^\varepsilon
    \prod_{p\leq k} \left( 1-\frac{\tfrac12 (|\lambda_f(p)|-1)^2+\tfrac12 (|\lambda_g(p)|-1)^2}{p} \right)
    + O_{\varepsilon} (k^{-1+\varepsilon}).
  \end{align*}
  This proves Lemma \ref{lemma:Poincare}.
\end{proof}

For $m=0$, $P_0(z,\Psi)=E(z|\Psi)$ is also called the incomplete Eisenstein series.
We write the Fourier expansion  as
\[
  E(z|\Psi) = a_{\Psi,0}(y) + \sum_{|\ell|\geq1} a_{\Psi,\ell}(y) e(\ell x).
\]
As in \cite[Lemma 4.6]{LMR}, we have the following lemma.

\begin{lemma}\label{lemma:FC}
  Let $M\geq1$.
  Let $\Psi:\mathbb{R}\rightarrow\mathbb{C}$ be a smooth compactly supported function such that $v^j \Psi^{(j)}(v) \ll_j M^j$ for all $j\in\mathbb{Z}_{\geq0}$ and $\supp \Psi\subset [1/2,B M]$ for some absolute constant $B>1$.
  Then for all $y>0$ we have
  \[
    a_{\Psi,0}(y) = \frac{3}{\pi} \tilde\Psi(-1) + \frac{1}{2\pi} \int_{\mathbb{R}} \tilde\Psi(-1/2-it)
    \left( y^{1/2+it} + \frac{\theta(1/2-it)}{\theta(1/2+it)} y^{1/2-it} \right) \dd t
  \]
  and for $\ell\neq0$,
  \[
    a_{\Psi,\ell}(y) \ll_{\varepsilon,A} \sqrt{y} \tau(|\ell|)  M^{2/3+\varepsilon}
    \left(\frac{M}{|\ell y|}\right)^{A} \left(1+\frac{1}{|\ell y|}\right)^\varepsilon ,
  \]
  for any $\varepsilon>0$ and $A\geq 0$.
\end{lemma}

\begin{proof}
  By Mellin inversion we have $E(z|\Psi) = \frac{1}{2\pi i} \int_{(2)} \tilde\Psi(-s) E(z,s) \dd s$. By \eqref{eqn:FE-ES}, we have
  \[
    a_{\Psi,0}(y) = \frac{1}{2\pi i} \int_{(2)} \tilde\Psi(-s)
    \left( y^s + \frac{\theta(1-s)}{\theta(s)} y^{1-s} \right) \dd s
  \]
  and  for $\ell\neq0$,
  \[
    a_{\Psi,\ell}(y) = \frac{1}{2\pi i} \int_{(2)} \tilde\Psi(-s)
   \frac{2\sqrt{y}}{\theta(s)}  \tau_{s-1/2}(|\ell|) K_{s-1/2}(2\pi |\ell|y)  \dd s.
  \]
  By shifting the contour to $\Re(s)=1/2$, we get
  \[
    a_{\Psi,0}(y) = \frac{3}{\pi} \tilde\Psi(-1) + \frac{1}{2\pi} \int_{\mathbb{R}} \tilde\Psi(-1/2-it)
    \left( y^{1/2+it} + \frac{\theta(1/2-it)}{\theta(1/2+it)} y^{1/2-it} \right) \dd t
  \]
  and  for $\ell\neq0$,
  \begin{align}\label{eqn:a_ell}
    a_{\Psi,\ell}(y)
    &  = \frac{\sqrt{y}}{\pi} \int_{\mathbb{R}} \tilde\Psi(-1/2-it)
   \frac{1}{\theta(1/2+it)}  \tau_{it}(|\ell|) K_{it}(2\pi |\ell|y)  \dd t \nonumber \\
   & \ll  \sqrt{y} \tau (|\ell|) \int_{\mathbb{R}} \Big|\tilde\Psi(-1/2-it)
   \frac{1}{\theta(1/2+it)}  K_{it}(2\pi |\ell|y) \Big|  \dd t.
  \end{align}
  Here $\frac{3}{\pi} \tilde\Psi(-1)$ comes from the residue at $s=1$.
%

  Note that for $t\gg1$ we have  $\theta(1/2+it) = \pi^{-1/2-it} \Gamma(1/2+it) \zeta(1+2it) \gg e^{-\pi|t|/2} (\log (1+|t|))^{-1}$.
  By  Balogh \cite{balogh1967asymptotic} (see also \cite[Eq. (4.7)]{LMR}), for $u>0$ and $t\in\mathbb{R}$ we have
  \begin{equation}\label{eqn:K<<1}
    K_{it}(u) \ll  u^{-1/2} e^{-u},
  \end{equation}
  and
  \begin{equation}\label{eqn:K<<2}
    K_{it}(u) \ll   |t|^{-1/3} e^{-\pi |t|/2}.
  \end{equation}
  By Holowinsky \cite[P. 1505]{Holowinsky}, for $u>0$ and $t\in\mathbb{R}$ we have
  \begin{equation}\label{eqn:K<<3}
    K_{it}(u) \ll |\Gamma(1/2+it)| (1+|t|)^{A+\varepsilon} u^{-A} (1+u^{-1})^\varepsilon,
  \end{equation}
  for any $\varepsilon>0$ and $A\geq0$.

  If $0< u=4\pi |\ell|y \ll1$, then by using \eqref{eqn:K<<3} in \eqref{eqn:a_ell} when $t\ll1$ and using \eqref{eqn:K<<2} when $t\gg1$, we obtain
  \begin{align*}
    a_{\Psi,\ell}(y)
    & \ll  \sqrt{y} \tau (|\ell|)
    \Big( \int_{t\ll1} (1+u^{-1})^\varepsilon \dd t + \int_{t\gg1} \frac{M^{2/3+\varepsilon}}{|t|^{2/3+\varepsilon}} t^{-1/3} \dd t \Big) \\
    &\ll \sqrt{y} \tau (|\ell|)  M^{2/3+\varepsilon} (1+u^{-1})^\varepsilon.
  \end{align*}
  If $u\gg1$, then by using \eqref{eqn:K<<3} in \eqref{eqn:a_ell} when $t\ll1$, using \eqref{eqn:K<<1} when $1\ll |t|\leq u/5$, and using \eqref{eqn:K<<2} when $|t|\geq u/5$,  we obtain
  \begin{align*}
    a_{\Psi,\ell}(y) &  \ll  \sqrt{y} \tau (|\ell|)
    \Big( \int_{t\ll1} u^{-A} \dd t
    + \int_{1\ll |t|\leq u/5} e^{\pi |t|/2} u^{-1/2} e^{-u} \dd t
    + \int_{|t|\geq u/5} \frac{M^A}{|t|^A} t^{-1/3} \dd t \Big)  \\
     & \ll \sqrt{y} \tau (|\ell|) \frac{M^A}{u^{A-2/3}} , \quad  A>2/3.
  \end{align*}
  This proves the lemma.
\end{proof}

\begin{remark}\label{rmk:a0}
  By \eqref{eqn:tildePsi<<} with $A=1+\varepsilon$, we obtain
  \begin{equation*}
    a_{\Psi,0}(y) = \frac{3}{\pi} \int_{0}^{\infty} \Psi(v) \frac{\dd v}{v^2} + O_\varepsilon (M^{1+\varepsilon} \sqrt{y}).
  \end{equation*}
  This is similar to previous works. But for the off-diagonal case we are considering, it seems not enough to use the above asymptotic formula. Our idea is to make use of the full expression of $a_{\Psi,0}$ as in Lemma \ref{lemma:FC}.
\end{remark}

\begin{lemma}\label{lemma:Eisenstein}
  Let $1\leq M\leq \log k$.
  Let $\Psi:\mathbb{R}\rightarrow\mathbb{C}$ be a smooth compactly supported function such that $y^\ell \Psi^{(\ell)}(y) \ll_\ell M^\ell$ for all $\ell\in\mathbb{Z}_{\geq0}$ and $\supp \Psi\subset [1/2,B M]$ for some absolute constant $B>1$.
  Then we have
  \[
    \langle E(\cdot|\Psi) F, G \rangle
    \ll_{\varepsilon}
    M^{5/3} (\log k)^\varepsilon
    \prod_{p\leq k} \left( 1-\frac{\tfrac14 (|\lambda_f(p)|-1)^2+\tfrac14 (|\lambda_g(p)|-1)^2}{p} \right).
  \]
\end{lemma}

\begin{proof}
  The proof uses the ideas from Holowinsky \cite{Holowinsky} and Lester--Matom\"aki--Radziwi{\l\l} \cite{LMR}.
  Let $1\leq Y\leq (\log k)^5$ be chosen later.  Define the smoothed incomplete Eisenstein series
  \begin{equation}\label{eqn:ES-Y}
    E^Y(z|h) = \sum_{\gamma\in\Gamma_\infty\backslash\Gamma} h(Y \Im(\gamma z)),
  \end{equation}
  where $h$ is a fixed nonnegative smooth compactly supported function on $\mathbb{R}$. Assume 
  \begin{equation}\label{eqn:h}
    \textrm{$\supp h\subset [1,2]$, $h^{(j)}(y) \ll_j 1$, and $\tilde{h}(-1)=\int_{0}^{\infty} h(y) y^{-2} \dd y=\pi/3$.}
  \end{equation}
  Define
  \begin{align*}
    I(Y):=\langle E^Y(\cdot |h) E(\cdot|\Psi) F, G\rangle .
  \end{align*}
  To estimate $\langle  E(\cdot|\Psi) F, G\rangle$, we will reduce it to bounding $I(Y)$.
  By the Mellin inversion,  we have
  \begin{align*}
    I(Y)
    = \frac{1}{2\pi i} \int_{(2)} \tilde{h}(-s) Y^s  \langle E(\cdot,s) E(\cdot|\Psi) F, G\rangle \dd s.
  \end{align*}
  Moving the line of integration to $\Re(s)=1/2$ we get
  \[
    I(Y) = \frac{3}{\pi} \tilde{h}(-1) Y \langle  E(\cdot|\Psi) F, G\rangle +  \frac{1}{2\pi i} \int_{(1/2)} \tilde{h}(-s) Y^s  \langle E(\cdot,s) E(\cdot|\Psi) F, G\rangle \dd s.
  \]
  By \eqref{eqn:HX} and unfolding, we have
  \begin{align*}
    \langle E(\cdot,1/2+it) E(\cdot|\Psi) F, G\rangle
    & = \int_{0}^{\infty} \int_{-1/2}^{1/2} \Psi(y) E(z,1/2+it) F(z) \overline{G(z)} \dd \mu z \\
    & \ll (1+|t|)^{1/2} \int_{1/2}^{BM}  \int_{-1/2}^{1/2} \sqrt{y} |F(z)| |G(z)| \dd \mu z .
  \end{align*}
  Using the Cauchy--Schwarz inequality, we get
  \begin{align*}
    \langle E(\cdot,1/2+it) E(\cdot|\Psi) F, G\rangle
    & \ll (1+|t|)^{1/2} \sqrt{M} \int_{1/2}^{BM}  \int_{-1/2}^{1/2}  ( |F(z)|^2 +  |G(z)|^2) \dd \mu z \\
    & \ll  (1+|t|)^{1/2} \sqrt{M} ( \|f\|_2^2 + \|g\|_2^2)
    \ll \sqrt{M(1+|t|)}.
  \end{align*}
  Hence
  \begin{equation}\label{eqn:IY=}
    \langle  E(\cdot|\Psi) F, G\rangle - I(Y) Y^{-1}
    \ll M^{1/2} Y^{-1/2}.
  \end{equation}

  Now we follow Holowinsky's approach to break $I(Y)$ into $I_\ell(Y)$.
  By the unfolding method, we obtain
  \begin{align*}
    I(Y)
    & = \int_{0}^{\infty} \int_{-1/2}^{1/2} h(Yy)  E(z|\Psi) F(z) \overline{G(z)} \dd \mu z \\
    &= a_f(1) \overline{a_g(1)} \sum_{\ell\in\mathbb{Z}} \sum_{\substack{n\geq1 \\ n+\ell\geq1}} \lambda_f(n) \lambda_g(n+\ell) (n(n+\ell))^{(k-1)/2}
    \\ & \hskip 90pt \cdot
     \int_{0}^{\infty}   h(Yy)  a_{\Psi,\ell}(y) \exp(-2\pi (2n+\ell) y) y^{k-2} \dd y
     =:  \sum_{\ell\in\mathbb{Z}}  I_\ell(Y),
  \end{align*}
  where
  \begin{align}\label{eqn:I0}
    I_0(Y)
    & = a_f(1) \overline{a_g(1)} \sum_{\substack{n\geq1  }} \lambda_f(n) \lambda_g(n) n^{k-1}
     \int_{0}^{\infty}   h(Yy)  a_{\Psi,0}(y) \exp(-4\pi n y) y^{k-2} \dd y ,
  \end{align}
  and, by Lemma \ref{lemma:FC}, for $\ell\neq0$,  we have
  \begin{align}
     \label{eqn:Iell}
    I_\ell(Y)
    & \ll  \frac{ M^{2/3+\varepsilon} Y^{-1/2+\varepsilon} \tau(|\ell|) \left(\frac{MY}{|\ell|}\right)^{A}}{L(1,\sym^2 f)^{1/2} L(1,\sym^2 g)^{1/2} }  \sum_{\substack{n\geq1 \\ n+\ell\geq1}} |\lambda_f(n) \lambda_g(n+\ell)| \nonumber
    \\ & \hskip 60pt \cdot \frac{(4\pi)^{k-1}}{\Gamma(k)}   (n(n+\ell))^{(k-1)/2}
     \int_{0}^{\infty}   h(Yy)   \exp(-2\pi (2n+\ell) y) y^{k-2} \dd y,
  \end{align}
  for any $A\geq0$.
  We can truncate the $n$-sum in $I_\ell(Y)$ at $n\leq k^{1+\varepsilon}$ with a negligible error $O_A((1+|\ell|)^{-2} k^{-A})$, since
  for $n\geq k^{1+\varepsilon}$ and $n+\ell\geq1$ we have
  \begin{align*}
    \frac{(4\pi)^{k-1}}{\Gamma(k)}  &  (n(n+\ell))^{(k-1)/2}
     \int_{0}^{\infty}   h(Yy)   \exp(-2\pi (2n+\ell) y) y^{k-2} \dd y \\
     &
     \ll \exp(k (\log n(n+\ell)) - 2\pi n/Y -2 \pi (n+\ell)/Y )
     \ll n^{-5} e^{-k^{1+\varepsilon/2}}
  \end{align*}
  for $1\leq Y\leq (\log k)^5$.

  We first estimate $I_\ell(Y)$ when  $|\ell|\geq k^\varepsilon$.
  For $n\leq k^{1+\varepsilon}$, similar to \eqref{eqn:integral} we can get
  \begin{multline*}
    \frac{(4\pi)^{k-1}}{\Gamma(k)}   (n(n+\ell))^{(k-1)/2}
     \int_{0}^{\infty}   h(Yy)   \exp(-2\pi (2n+\ell) y) y^{k-2} \dd y \\
     =
     \frac{(4\pi)^{k-1}}{\Gamma(k)}   (n(n+\ell))^{(k-1)/2}
     \int_{0}^{\infty}  \frac{1}{2\pi i} \int_{(2)} \tilde{h}(-s) (yY)^{s} \dd s \exp(-2\pi (2n+\ell) y) y^{k-2} \dd y \\
     \ll
     \frac{(4\pi)^{k-1}}{\Gamma(k)}   (n(n+\ell))^{(k-1)/2}
         Y^2 \int_{0}^{\infty} \exp(-2\pi (2n+\ell) y) y^{k} \dd y \\
     \ll \frac{(4\pi)^{k-1}}{\Gamma(k)}   (n+\ell/2)^{k-1}
         Y^2  \frac{\Gamma(k+1)}{(2\pi)^{k+1} (2n+\ell)^{k+1}}
     \ll n^{-2} k Y^2.
  \end{multline*}
  Here we have used the fact $n^{1/2} (n+\ell)^{1/2}\leq n+\ell/2$.
  Hence by taking $A$ to be large enough in \eqref{eqn:Iell}, we get
  \begin{equation}\label{eqn:Iell>>}
    \sum_{|\ell|\geq k^{\varepsilon}}  I_\ell(Y) \ll_{\varepsilon,A} k^{-A}.
  \end{equation}

  For $1\leq |\ell | \leq k^\varepsilon$, we can deal with the integrals the same as in the proof of Lemma \ref{lemma:Poincare}. By \cite[Eq. (4.21)]{LMR}, for $n\leq k^{1+\varepsilon}$ we have
  \begin{multline}\label{eqn:integral=}
    \frac{(4\pi)^{k-1}}{\Gamma(k)}   (n(n+\ell))^{(k-1)/2}
     \int_{0}^{\infty}   h(Yy)   \exp(-2\pi (2n+\ell) y) y^{k-2} \dd y
     \\
     = \frac{1}{k-1} h\left(\frac{Y(k-1)}{4\pi (n+\ell/2)}\right) +O(k^{\varepsilon-2}) + O(n^{-3/2} k^{\varepsilon-1/2}).
  \end{multline}
  Hence we get
  \[
    I_\ell(Y)
    \ll  \frac{ M^{2/3+\varepsilon} Y^{-1/2+\varepsilon} \tau(|\ell|) \left(\frac{MY}{|\ell|}\right)^{A}}{k L(1,\sym^2 f)^{1/2} L(1,\sym^2 g)^{1/2} }  \sum_{\substack{1\leq n\leq k Y}} |\lambda_f(n) \lambda_g(n+\ell)|
    + O_\varepsilon( k^{-1/2+\varepsilon} ),
  \]
  for any $A\geq0$.
  By similar argument as in \eqref{eqn:Holowinsky} we get
  \[
    I_\ell(Y)
    \ll  \frac{ M^{2/3} Y^{1/2} \tau(|\ell|) \left(\frac{MY}{|\ell|}\right)^{A}}{L(1,\sym^2 f)^{1/2} L(1,\sym^2 g)^{1/2} }  (\log k)^{\varepsilon} \prod_{p\leq k}
     \left(1+\frac{|\lambda_f(p)|-1}{p} \right) \left(1+\frac{|\lambda_g(p)|-1}{p} \right).
  \]
  Together with \eqref{eqn:HS} we obtain
  \begin{align*}
    \sum_{1\leq |\ell|\leq k^{\varepsilon}}  I_\ell(Y) &
    \ll_{\varepsilon}
    M^{2/3} Y^{1/2} (\log k)^\varepsilon
    \prod_{p\leq k} \left( 1-\frac{\tfrac12 (|\lambda_f(p)|-1)^2+\tfrac12 (|\lambda_g(p)|-1)^2}{p} \right) \\
    & \hskip 120pt  \cdot \sum_{1\leq |\ell|\leq k^{\varepsilon}} \tau(|\ell|) \left(\frac{MY}{|\ell|}\right)^{A} .
  \end{align*}
  By taking $A=1+\varepsilon$, we get
  \begin{align}\label{eqn:Iell<<}
    \sum_{1\leq |\ell|\leq k^{\varepsilon}}  I_\ell(Y)
    \ll_{\varepsilon}
    M^{5/3} Y^{3/2} (\log k)^\varepsilon
    \prod_{p\leq k} \left( 1-\frac{\tfrac12 (|\lambda_f(p)|-1)^2+\tfrac12 (|\lambda_g(p)|-1)^2}{p} \right).
  \end{align}

  Now we deal with $I_0(Y)$. We will sue the explicit formula for the zeroth Fourier coefficient $a_{\Psi,0}(y)$.
  By Lemma \ref{lemma:FC}, we have
   \begin{align*}
    I_{0}(Y)
    & = \frac{3}{\pi} \tilde{\Psi}(-1) a_f(1) \overline{a_g(1)}
    \sum_{\substack{n\geq1}} \lambda_f(n) \lambda_g(n) n^{k-1}
     \int_{0}^{\infty}   h(Yy) \exp(-4\pi n y) y^{k-2} \dd y  \nonumber \\
     & \qquad  +
      \frac{1}{2\pi} \int_{\mathbb{R}} \tilde\Psi(-1/2-it)
      a_f(1) \overline{a_g(1)}
    \sum_{\substack{n\geq1}} \lambda_f(n) \lambda_g(n) \nonumber \\
    & \qquad \quad \cdot  n^{k-1}
     \int_{0}^{\infty}   h(Yy) \exp(-4\pi n y) y^{k-2}
    \left( y^{1/2+it} + \frac{\theta(1/2-it)}{\theta(1/2+it)} y^{1/2-it} \right) \dd y \dd t.
  \end{align*}
  Define
  \[
      S(t,H) :=  a_f(1) \overline{a_g(1)}
    \sum_{\substack{n\geq1}} \lambda_f(n) \lambda_g(n)
    n^{k-1}
     \int_{0}^{\infty}   H(Yy) \exp(-4\pi n y) y^{k-2} y^{it} \dd y.
  \]
  Then we have
  \begin{equation}\label{eqn:I0=S}
    I_{0}(Y) = \frac{3}{\pi} \tilde{\Psi}(-1) S(0,h) +
     \frac{1}{2\pi Y^{1/2}} \int_{\mathbb{R}} \tilde\Psi(-1/2-it) \left( S(t,h_1)+ \frac{\theta(1/2-it)}{\theta(1/2+it)} S(-t,h_1)\right) \dd t,
  \end{equation}
  where $h_1(y)=h(y) y^{1/2}$.
  It suffices to bound $S(t,H)$ with $t\in \mathbb{R}$ and smooth functions $H: \mathbb{R} \rightarrow \mathbb{R}_{\geq0}$ such that $\supp H\subset [1,2]$ and $H^{(j)}\ll_j 1$.

  By \eqref{eqn:integral=} we have
  \begin{align*}
    S(t,H) & \ll \frac{1}{L(1,\sym^2 f)^{1/2} L(1,\sym^2 g)^{1/2}}
    \sum_{\substack{n\geq1}} |\lambda_f(n) \lambda_g(n)|
    \\
    & \hskip 60pt\cdot
    n^{k-1} \frac{(4\pi)^{k-1}}{\Gamma(k)}
     \int_{0}^{\infty}   H(Yy) \exp(-4\pi n y) y^{k-2} \dd y \\
     & \ll \frac{1}{L(1,\sym^2 f)^{1/2} L(1,\sym^2 g)^{1/2}}
    \sum_{\substack{n\asymp kY}} |\lambda_f(n) \lambda_g(n)|
    \frac{1}{k} + O(k^{-1/4})
    \ll (\log k)^{100}.
  \end{align*}
  Recall that  $\tilde\Psi$ satisfies \eqref{eqn:tildePsi<<}.
  Hence the contribution from $|t|\geq M (\log k)^\varepsilon$ to $I_0(Y)$
  is bounded by
  \begin{multline} \label{eqn:t-large}
    O_A\left(  \int_{M (\log k)^\varepsilon}^{\infty} M^{A+1} |t|^{-A-1} (\log k)^{100} \dd t  \right) \\
    = O_A\left(    M^{A+1}  (\log k)^{100} M^{-A} (\log k)^{-A\varepsilon } \right)
    = O_{\varepsilon,A} \left(  (\log k)^{-A} \right).
  \end{multline}

  Now we consider the case $|t|\leq M (\log k)^\varepsilon$. By our assumption, we have $|t| \leq (\log k)^2$.
  Mellin inversion gives
  \begin{multline*}
     n^{k-1} \frac{(4\pi)^{k-1}}{\Gamma(k)}
     \int_{0}^{\infty}   H(Yy) \exp(-4\pi n y) y^{k-2} y^{it} \dd y
     \\
     =
     \frac{1}{2\pi i} \int_{(2)} \tilde{H}(-s) Y^s (4\pi n )^{-s-it}
     \frac{\Gamma(k+s+it-1) }{\Gamma(k)}    \dd s.
  \end{multline*}
  Hence
  \[
    S(t,H) \ll \frac{| \sum_{\substack{n\geq1}} \lambda_f(n) \lambda_g(n) \frac{1}{2\pi i} \int_{(2)} \tilde{H}(-s) Y^s (4\pi n )^{-s-it}
     \frac{\Gamma(k+s+it-1) }{\Gamma(k)}    \dd s  |}{L(1,\sym^2 f)^{1/2} L(1,\sym^2 g)^{1/2}}.
  \]
  Recall that $\sum_{n\geq1} \frac{\lambda_f(n) \lambda_g(n)}{n^s} = \frac{L(s,f\times g)}{\zeta(2s)}$. We have
  \[
    S(t,H) \ll \frac{| \frac{1}{2\pi i} \int_{(2)} \tilde{H}(-s) Y^s (4\pi)^{-s} \frac{L(s+it,f\times g)}{\zeta(2s+2it)}
     \frac{\Gamma(k+s+it-1) }{\Gamma(k)}    \dd s  |}{L(1,\sym^2 f)^{1/2} L(1,\sym^2 g)^{1/2}}.
  \]
  Moving the line of integration to $\Re(s)=1/2$ and noting that $L(s,f\times g)$ is holomorphic if $f\neq g$, we obtain
  \[
    S(t,H) \ll \frac{ Y^{1/2} k^{-1/2} }{L(1,\sym^2 f)^{1/2} L(1,\sym^2 g)^{1/2}}
     \int_{\mathbb{R}} |\tilde{H}(-1/2-iv)| |\frac{L(1/2+iv+it,f\times g)}{\zeta(1+2iv+2it)}|    \dd v.
  \]
  Here we have used the fact $
     \frac{|\Gamma(k+iv+it-1/2)|}{\Gamma(k)} \leq   \frac{|\Gamma(k-1/2)|}{\Gamma(k)} \ll k^{-1/2} $.
  Now we should apply Soundararajan's weak subconvexity bounds for $L(1/2+iv+it,f\times g)$, since the Ramanujan conjecture is known in this setting. The analytic conductor of this $L$-function is of size $k^2 (1+|v+t|)^2$. By Deligne's bounds \cite{Deligne}, Soundararajan's theorem \cite{Soundararajan2010weak} gives
  \[
    L(1/2+iv+it,f\times g) \ll \frac{k^{1/2} (1+|v+t|)^{1/2}}{(\log k)^{1-\varepsilon}}.
  \]
  Using the bounds $|\tilde{H}(-1/2-iv)|\ll (1+|v|)^{-2022}$ and $|\zeta(1+2iv+2it)|^{-1}\ll \log (1+|v+t|)$, we get
  \begin{align}\label{eqn:t-small}
    S(t,H) & \ll \frac{ Y^{1/2} k^{-1/2} }{L(1,\sym^2 f)^{1/2} L(1,\sym^2 g)^{1/2}}
     \int_{\mathbb{R}} (1+|v|)^{-2022} \frac{k^{1/2} (1+|v+t|)^{1/2}}{(\log k)^{1-\varepsilon}}    \dd v  \nonumber \\
     & \ll   \frac{ Y^{1/2}  }{L(1,\sym^2 f)^{1/2} L(1,\sym^2 g)^{1/2}}
      \frac{ (1+|t|)^{1/2}}{(\log k)^{1-\varepsilon}}  .
  \end{align}
  By \eqref{eqn:I0=S}, \eqref{eqn:t-large}, and \eqref{eqn:t-small} we get
  \begin{align}\label{eqn:I0<<}
    I_0(Y) &  \ll   \frac{ Y^{1/2}  }{L(1,\sym^2 f)^{1/2} L(1,\sym^2 g)^{1/2} (\log k)^{1-\varepsilon}}
       \nonumber \\
     & \qquad  + Y^{-1/2} \int_{|t|\leq M(\log k)^\varepsilon} \frac{ Y^{1/2}  }{L(1,\sym^2 f)^{1/2} L(1,\sym^2 g)^{1/2}}
      \frac{ (1+|t|)^{1/2}}{(\log k)^{1-\varepsilon}} \dd t \nonumber \\
     & \ll  \frac{ Y^{1/2} +  M^{1/2}  (\log k)^\varepsilon }{L(1,\sym^2 f)^{1/2} L(1,\sym^2 g)^{1/2} (\log k)^{1-\varepsilon}}
     \ll ( Y^{1/2} +  M^{1/2} )  (\log k)^\varepsilon .
  \end{align}

  Now by \eqref{eqn:IY=}, \eqref{eqn:Iell>>}, \eqref{eqn:Iell<<}, and \eqref{eqn:I0<<}, we have
  \begin{multline*}
    \langle  E(\cdot|\Psi) F, G\rangle
    \ll M^{1/2} Y^{-1/2} (\log k)^\varepsilon \\
    +
    M^{5/3} Y^{1/2} (\log k)^\varepsilon
    \prod_{p\leq k} \left( 1-\frac{\tfrac12 (|\lambda_f(p)|-1)^2+\tfrac12 (|\lambda_g(p)|-1)^2}{p} \right).
  \end{multline*}
  By taking $$Y=\prod_{p\leq k} \left( 1+\frac{\tfrac12 (|\lambda_f(p)|-1)^2+\tfrac12 (|\lambda_g(p)|-1)^2}{p} \right),$$
  we complete the proof of Lemma \ref{lemma:Eisenstein}.
\end{proof}

\begin{proposition} \label{prop:Holowinsky}
  With the notation and assumptions of Theorem \ref{thm:decor}, we have
  \[
    \langle \psi F, G \rangle
    = O_\varepsilon \left( M^{5/3} (\log k)^\varepsilon
    \prod_{p\leq k} \left( 1-\frac{\tfrac14 (|\lambda_f(p)|-1)^2+\tfrac14 (|\lambda_g(p)|-1)^2}{p} \right) \right) .
  \]
\end{proposition}

\begin{proof}
  This is a consequence of \eqref{eqn:<psiFG><<PS}, Lemmas \ref{lemma:Poincare} and \ref{lemma:Eisenstein}.
\end{proof}

\section{Proof of Theorem \ref{thm:decor}} \label{sec:proof}

Now we are ready to prove Theorem \ref{thm:decor} by using Propositions \ref{prop:Sound} and \ref{prop:Holowinsky} and an optimization technique.

\begin{proof}[Proof of Theorem \ref{thm:decor}]
  By Propositions \ref{prop:Sound} and \ref{prop:Holowinsky}, we have
  \begin{multline*}
    \langle \psi F, G \rangle
    \ll_\varepsilon   M^{5/3} (\log k)^\varepsilon
    \min\Big(
       \prod_{p\leq k} \left(1 - \frac{\tfrac12 \lambda_f(p^2) + \tfrac12  \lambda_g(p^2)+\delta_1}{p} \right), \\
       \prod_{p\leq k} \left( 1-\frac{\tfrac14 (|\lambda_f(p)|-1)^2+\tfrac14 (|\lambda_g(p)|-1)^2}{p} \right) \Big).
  \end{multline*}
  So for any $\alpha\in[0,1]$, we have
  \begin{align*}
    \langle \psi F, G \rangle
    & \ll_\varepsilon   M^{5/3} (\log k)^\varepsilon
    \prod_{p\leq k} \left(1 - \frac{ \tfrac12 \lambda_f(p^2) + \tfrac12  \lambda_g(p^2)+\delta_1 }{p} \right)^\alpha  \\
    & \hskip 120pt \cdot \left( 1-\frac{ \tfrac14 (|\lambda_f(p)|-1)^2+\tfrac14 (|\lambda_g(p)|-1)^2 }{p} \right)^{1-\alpha} .
  \end{align*}
  Since for any sequence $b_p\ll1$ we have  $ \prod_{p\leq k} (1-\frac{b_p}{p})^{\alpha} \ll \prod_{p\leq k} (1-\frac{\alpha b_p}{p}) $, we obtain
  \begin{align*}
    \langle \psi F, G \rangle
    & \ll_\varepsilon   M^{5/3} (\log k)^\varepsilon
    \prod_{p\leq k} \left(1 - \frac{\alpha( \tfrac12 \lambda_f(p^2) + \tfrac12  \lambda_g(p^2)+\delta_1 )}{p} \right)  \\
    & \hskip 120pt \cdot \left( 1-\frac{(1-\alpha)( \tfrac14 (|\lambda_f(p)|-1)^2+\tfrac14 (|\lambda_g(p)|-1)^2)}{p} \right) \\
    & \ll_\varepsilon  M^{5/3} (\log k)^\varepsilon
    \prod_{p\leq k} \left(1 + \frac{L_p(\alpha)}{p} \right),
  \end{align*}
  where
  \[
    L_p(\alpha) := - \alpha( \tfrac12 \lambda_f(p^2) + \tfrac12  \lambda_g(p^2)+\delta_1 )
    - (1-\alpha)( \tfrac14 (|\lambda_f(p)|-1)^2+\tfrac14 (|\lambda_g(p)|-1)^2).
  \]
  By the Hecke relations, we have $\lambda_f(p^2)=\lambda_f(p)^2-1$. Now write $\lambda_1=|\lambda_f(p)|$ and $\lambda_2=|\lambda_g(p)|$. We have $\lambda_j\in[0,2]$ and $L_p(\alpha)=L(\alpha,\lambda_1,\lambda_2)$ with
  \[
    L(\alpha,\lambda_1,\lambda_2)
    =   -  \alpha \delta_1
    - \frac{1+\alpha}{4} \lambda_1^2 + \frac{1-\alpha}{2} \lambda_1
    - \frac{1+\alpha}{4} \lambda_2^2 + \frac{1-\alpha}{2} \lambda_2 - \frac{1-3\alpha}{2}.
  \]
  We want to find
  \[
    L = \min_{\alpha\in [0,1]} \max_{\substack{\lambda_1\in[0,2] \\ \lambda_2\in[0,2]}} L(\alpha,\lambda_1,\lambda_2).
  \]
  Note that
  \[
    L(\alpha,\lambda_1,\lambda_2)
    =   -  \alpha \delta_1 - \frac{1-3\alpha}{2} + \frac{(1-\alpha)^2}{2(1+\alpha)}
    - \frac{1+\alpha}{4} \left(\lambda_1 - \frac{1-\alpha}{1+\alpha} \right)^2
    - \frac{1+\alpha}{4} \left(\lambda_2 - \frac{1-\alpha}{1+\alpha} \right)^2.
  \]
  We have
  \[
    L = \min_{\alpha\in [0,1]} L(\alpha),
  \]
  where
   \[
     L(\alpha) := -  \alpha \delta_1 - \frac{1-3\alpha}{2} + \frac{(1-\alpha)^2}{2(1+\alpha)}
     = (2-\delta_1)(1+\alpha) + \frac{2}{1+\alpha} -4+\delta_1.
   \]
   Assume $\delta_1\in[0,1]$. Then we have
   \[
     L=2\sqrt{2(2-\delta_1)}-4+\delta_1.
   \]
   Hence we obtain
   \[
     \langle \psi F, G \rangle
     \ll_\varepsilon  M^{5/3} (\log k)^\varepsilon
    \prod_{p\leq k} \left(1 + \frac{L}{p} \right)
    \ll_\varepsilon  M^{5/3} (\log k)^{L+\varepsilon}.
   \]
   This completes the proof of Theorem \ref{thm:decor}.
\end{proof}

\section{Proof of Theorem \ref{thm:QUE}} 
\label{sec:QUE}

\begin{proof}[Proof of Theorem \ref{thm:QUE}]
  By Lester, Matom\"aki, and Radziwi{\l\l} \cite[Theorem 1.3]{LMR}, for $f_j\in H_k$,
  we have
  \[
    \langle \psi F_j, F_j \rangle = \frac{3}{\pi} \langle \psi,1\rangle + O_\varepsilon( M^2 (\log k)^{-0.007}).
  \]
  By Theorem \ref{thm:decor}, for $i\neq j$, we have
  \[
    \langle \psi F_i, F_j \rangle =  O_\varepsilon( M^{5/3} (\log k)^{-\delta+\varepsilon}).
  \]
  Note that for $f=\sum_{1\leq j\leq J} c_j f_j$ and $\|f\|_2=1$, we have $\sum_{1\leq j\leq J}|c_j|^2=1$. Hence
  \begin{align*}
    \langle \psi , |F|^2 \rangle & = \langle \psi F, F \rangle
    = \langle \psi \sum_{1\leq i\leq J} c_i F_i, \sum_{1\leq j\leq J} c_j F_j \rangle \\
    & = \sum_{1\leq i\leq J} \sum_{1\leq j\leq J} c_i \overline{c_j} \langle \psi F_i, F_j \rangle \\
    & = \sum_{1\leq j\leq J } |c_j|^2 \langle \psi F_j, F_j \rangle
    + \sum_{1\leq i\leq J} \sum_{\substack{1\leq j\leq J \\ j\neq i}} c_i \overline{c_j} \langle \psi F_i, F_j \rangle \\
    & = \frac{3}{\pi} \langle \psi,1\rangle + O( M^2 (\log k)^{-0.007})
    + O_\varepsilon\bigg( \sum_{1\leq i\leq J} \sum_{\substack{1\leq j\leq J \\ j\neq i}} |c_i \overline{c_j}|  M^{5/3} (\log k)^{-\delta+\varepsilon} \bigg).
  \end{align*}
  Now we use the fact $|c_i \overline{c_j}| \leq |c_i|^2+ |c_j|^2$, getting
  \[
    \langle \psi , |F|^2 \rangle
    = \frac{3}{\pi} \langle \psi,1\rangle + O( M^2 (\log k)^{-0.007})
    + O_\varepsilon\bigg( J M^{5/3} (\log k)^{-\delta+\varepsilon} \bigg).
  \]
  Note that we have the trivial bound $\langle \psi , |F|^2 \rangle  \ll \langle 1, |F|^2 \rangle =1 $.
  This completes the proof of Theorem \ref{thm:QUE}.
\end{proof}



\section{Proof of Theorem \ref{thm:equidistribution}} \label{sec:zero}

To prove Theorem \ref{thm:equidistribution}, we only need to prove that 
the bulk of zeros of $f\in H_k^{(J)}$ lies in the fundamental domain. 
It suffices to prove that the order of $f$ at infinity $\nu_\infty(f)$ is $o(k)$.
Write the Fourier expansion of $f\in H_k^{(J)}$ as
\[
  f(z) = \sum_{n\geq1} a_f(n) n^{\frac{k-1}{2}}e(nz).
\]
So we need to show that there exists $n=o(k)$ such that $a_f(n)\neq0$. 

Recall that we have the incomplete Eisenstein series $E^{1/Y}(z|h)$ as in \eqref{eqn:ES-Y}, where $Y\geq1$ is a parameter which will be chosen later.
We assume the same conditions \eqref{eqn:h} on $h$.   By the proof of Theorem \ref{thm:decor} we know that EQQUE holds for $f$ with the observable $\psi(z)=E^{1/Y}(z|h)$ by taking $M=Y$. Assume $\|f\|_2=1$. As in Theorem \ref{thm:QUE}, we have
\begin{equation*}
  \langle E^{1/Y}(\cdot|h) F,F \rangle =  \frac{3}{\pi} \langle E^{1/Y}(\cdot|h) ,1 \rangle 
  + O_{\varepsilon}( J Y^{5/3} (\log k)^{-\delta+\varepsilon}).
\end{equation*}
Note that 
$\frac{3}{\pi}\langle E^{1/Y}(\cdot|h) ,1 \rangle  =1/Y$. 
By taking 
\[
  Y= \left(\frac{(\log k)^{\delta-\varepsilon}}{J}\right)^{3/8},
\]
we get $\langle E^{1/Y}(\cdot|h) F,F \rangle\sim 1/Y$.

By the unfolding method as in \S \ref{sec:Holowinsky}, we have 
\[
  \langle E^{1/Y}(\cdot|h) F,F \rangle
  = \sum_{n\geq1} |a_f(n)|^2 n^{k-1} \int_{0}^{\infty} h(y/Y) e^{-4\pi ny} y^{-2} \dd y. 
\]
We want to show that the contribution from the terms $n\geq k Y^{-1/2}$ is negligibly small. 
Let $f=\sum_{j=1}^J c_j f_j$ with $f_j\in H_k$. Then we have $\sum_j |c_j|^2=1$ and $a_f(n) = \sum_j c_j a_{f_j}(1)\lambda_{f_j}(n)$. 
So 
\[
  |a_f(n)|^2 \leq \sum_j |a_{f_j}(1)\lambda_{f_j}(n)|^2 \ll \tau(n)^2 J \frac{(4\pi)^{k-1}}{\Gamma(k)} \log k.
\]
Here we have used \eqref{eqn:firstcoeff} and the fact $L(1,\sym^2 f_j) \gg 1/\log k$ in \cite{HL}.
By Stirling's formula, we have 
\begin{align*}
  \frac{(4\pi)^{k-1}}{\Gamma(k)} n^{k-1} e^{-4\pi nY} 
  & \ll k^{-1/2} \exp\left(- (k-1)\log \frac{k-1}{4\pi e n} - 4\pi nY \right).
\end{align*}
Assume $J\leq (\log k)^{\delta-2\varepsilon}$. So that $Y\geq (\log k)^{\varepsilon}$.
For $kY^{-1/2} \leq n \leq k (\log k)^2 $, we have 
\begin{align*}
  \frac{(4\pi)^{k-1}}{\Gamma(k)} n^{k-1} e^{-4\pi nY}
  & \ll k^{-1/2} \exp\left(3(k-1)\log \log k - 4\pi kY^{1/2} \right)  \ll \exp(-4\pi k).
\end{align*}
For $k(\log k)^2 \leq n \leq k^2$, we have
\begin{align*}
  \frac{(4\pi)^{k-1}}{\Gamma(k)} n^{k-1} e^{-4\pi nY}
  & \ll k^{-1/2} \exp\left(2(k-1)\log k - 4\pi k(\log k)^2  \right)  \ll \exp(-4\pi k).
\end{align*}
For $n\geq k^2$, we have 
\begin{align*}
  \frac{(4\pi)^{k-1}}{\Gamma(k)} n^{k-1} e^{-4\pi nY}
  \ll \exp(-4\pi k -n).
\end{align*}
Hence we get 
\begin{align*}
  \sum_{n\geq kY^{-1/2} } & |a_f(n)|^2 n^{k-1} \int_{0}^{\infty} h(y/Y) e^{-4\pi ny} y^{-2} \dd y  \\
   & \ll \sum_{k(\log k)^2 \leq n \leq k^2} \tau(n)^2 J (\log k)  Y^{-1} \exp(-4\pi k) \\
   & \qquad  + \sum_{ n \geq k^2} \tau(n)^2 J (\log k)  Y^{-1} \exp(-4\pi k -n)
   \\
   & \ll e^{-k}. 
\end{align*}
So we have 
\[\sum_{n\leq kY^{-1/2} } |a_f(n)|^2 n^{k-1} \int_{0}^{\infty} h(y/Y) e^{-4\pi ny} y^{-2} \dd y \sim 1/Y.\]
Therefore, there exists $n=o(k)$ satisfying that $a_f(n)\neq 0$, and hance $\nu_\infty(f)=o(k)$. 
Together with \cite[Theorem 1.2]{Rudnick} and Theorem \ref{thm:QUE}, we complete the proof of Theorem \ref{thm:equidistribution}.

\section*{Acknowledgements}
The author wants to thank  Peter Humphries, Noam Kimmel, Stephen Lester, Ze\'ev Rudnick, Jesse Thorner for valuable discussions on this topic.
He'd like to thank Jianya Liu and Ze\'ev Rudnick for constant encouragement and Stephen Lester for sending him Iwaniec's note.
He also wants to thank the referees for their very helpful comments and suggestions.


\end{document}